\documentclass{amsart}

\usepackage{amssymb}
\usepackage{amsmath}
\usepackage{amsthm}
\usepackage{amsfonts}

\usepackage{a4wide}
\usepackage{longtable}
\usepackage{multirow}

\newtheorem{theorem}{Theorem}[section]
\newtheorem{proposition}{Proposition}[section]
\newtheorem{corollary}{Corollary}[section]
\newtheorem{lemma}{Lemma}[section]
\newtheorem{conjecture}{Conjecture}[section]

\theoremstyle{definition}
\newtheorem{definition}{Definition}[section]
\newtheorem{remark}{Remark}[section]

\numberwithin{equation}{section}

\begin{document}


\title{Successive approximation of \(p\)-class towers}

\author{Daniel C. Mayer}
\address{Naglergasse 53 \\ 8010 Graz \\ Austria}
\email{algebraic.number.theory@algebra.at}
\urladdr{http://www.algebra.at}

\thanks{Research supported by the Austrian Science Fund (FWF): P 26008-N25}


\date{October 11, 2017}


\begin{abstract}
Let \(F\) be a number field and \(p\) be a prime.
In the successive approximation theorem, we prove that,
for each integer \(n\ge 1\), finitely many candidates for
the Galois group \(\mathrm{G}_p^n{F}\) of the \(n\)th stage \(F_p^{(n)}\)
of the \(p\)-class tower \(F_p^{(\infty)}\) over \(F\)
are determined by abelian type invariants of \(p\)-class groups
\(\mathrm{Cl}_p{E}\) of unramified extensions \(E/F\) with degree \(\lbrack E:F\rbrack=p^{n-1}\).
Illustrated by the most extensive numerical results available currently,
the transfer kernels \(\ker(T_{F,E})\) of the \(p\)-class extensions
\(T_{F,E}:\mathrm{Cl}_p{F}\to\mathrm{Cl}_p{E}\)
from \(F\) to unramified cyclic degree-\(p\) extensions \(E/F\)
are shown to be capable of narrowing down the number of contestants significantly.
By determining the isomorphism type of the maximal subgroups \(S<G\)
of all \(3\)-groups \(G\) with coclass \(\mathrm{cc}(G)=1\),
and establishing a general theorem on the connection between the \(p\)-class towers
of a number field \(F\) and of an unramified abelian \(p\)-extensions \(E/F\),
we are able to provide a theoretical proof of the realization
of certain \(3\)-groups \(S\) with maximal class by \(3\)-tower groups \(\mathrm{G}_3^\infty{E}\)
of dihedral fields \(E\) with degree \(6\),
which could not be realized up to now.
\end{abstract}

\maketitle



\section{Introduction}
\label{s:Intro}
\noindent
For a prime number \(p\) and an algebraic number field \(F\),
let \(F_p^{(\infty)}\) be the \(p\)-\textit{class tower},
more precisely the unramified Hilbert \(p\)-class field tower,
that is the maximal unramified pro-\(p\) extension, of \(F\).
The individual stages \(F_p^{(n)}\) of the tower
\[F=F_p^{(0)}\le F_p^{(1)}\le F_p^{(2)}\le\ldots\le F_p^{(n)}\le\ldots\le F_p^{(\infty)}\]
are described by the derived quotients
\(\mathfrak{G}/\mathfrak{G}^{(n)}\simeq\mathrm{G}_p^n{F}:=\mathrm{Gal}(F_p^{(n)}/F)\) with \(n\ge 1\),
of the \(p\)-\textit{class tower group} \(\mathfrak{G}:=\mathrm{G}_p^\infty{F}:=\mathrm{Gal}(F_p^{(\infty)}/F)\).
The purpose of this paper is to report on the
most up-to-date theoretical view of \(p\)-class towers
and the state of the art of actual numerical investigations.
After a summary of algebraic and arithmetic foundations in \S\
\ref{s:Foundations},
four crucial concepts will illuminate
recent innovation and progress in a very ostensive way:
\begin{itemize}
\item
the \textit{Artin limit pattern} \((\tau^{(\infty)}{F},\varkappa^{(\infty)}{F})\) of the \(p\)-class tower \(F_p^{(\infty)}\) in \S\
\ref{s:LimitPattern},
\item
\textit{successive approximation} and the current status of computational perspectives in \S\
\ref{s:Approximation},
\item
\textit{maximal subgroups} of \(3\)-class tower groups with coclass one in \S\
\ref{s:MaximalSubgroups},
and
\item
the realization of new \(3\)-class tower groups over \textit{dihedral fields} in \S\
\ref{s:General}.
\end{itemize}



\section{Algebraic and arithmetic foundations}
\label{s:Foundations}

\subsection{Abelian type invariants}
\label{ss:ATI}
\noindent
First, we recall the concepts of
abelian type invariants and abelian quotient invariants
in the context of finite \(p\)-groups and infinite pro-\(p\) groups,
and we specify our conventions in their notation.

Let \(p\ge 2\) be a prime number.
It is well known that
a finite abelian group \(A\)
with order \(\lvert A\rvert\) a power of \(p\)
possesses a \textit{unique} representation
\begin{equation}
\label{eqn:ATI}
A\simeq\oplus_{i=1}^s(\mathbb{Z}/p^{e_i}\mathbb{Z})^{r_i}
\end{equation}
as a direct sum
with integers \(s\ge 0\),
\(r_i\ge 1\) for \(1\le i\le s\),
and \textit{strictly decreasing} \(e_1>\ldots >e_s\ge 1\).


\begin{definition}
\label{dfn:ATI}
The \textit{abelian type invariants} of \(A\)
are given either in \textit{power form},
\begin{equation}
\label{eqn:PowATI}
\mathrm{ATI}(A):=
\lbrack\overbrace{p^{e_1},\ldots,p^{e_1}}^{r_1 \text{ times}},\ldots,
\overbrace{p^{e_s},\ldots,p^{e_s}}^{r_s \text{ times}}\rbrack,
\end{equation}
or in \textit{logarithmic form} with formal exponents indicating iteration,
\begin{equation}
\label{eqn:LogATI}
\mathrm{ATI}(A):=\lbrack e_1^{r_1},\ldots,e_s^{r_s}\rbrack.
\end{equation}
\end{definition}


Let \(G\) be a pro-\(p\) group
with commutator subgroup \(G^\prime\)
and \textit{finite} abelianization \(G^{ab}:=G/G^\prime\).

\begin{definition}
\label{dfn:AQI}
The \textit{abelian quotient invariants} of \(G\)
are the abelian type invariants of the biggest abelian quotient of \(G\)
\begin{equation}
\label{eqn:AQI}
\mathrm{AQI}(G):=\mathrm{ATI}(G^{ab}).
\end{equation}
\end{definition}


\subsubsection{Higher abelian quotient invariants of a pro-\(p\) group}
\label{sss:TauGrp}
\noindent
Within the frame of group theory,
abelian quotient invariants of \textit{higher order}
are defined recursively in the following manner.

\begin{definition}
\label{dfn:TauGrp}
The set of all maximal subgroups of \(G\)
which contain the commutator subgroup,
\begin{equation}
\label{eqn:LyrGrp}
\mathrm{Lyr}_1{G}:=\lbrace S\triangleleft G\mid G^\prime\le S,\ (G:S)=p\rbrace,
\end{equation}
is called the \textit{first layer} of subgroups of \(G\).
For any positive integer \(n\ge 1\),
\textit{abelian quotient invariants of \(n\)th order} of \(G\)
are defined recursively by 
\begin{equation}
\label{eqn:TauGrp}
\tau^{(1)}{G}:=\mathrm{AQI}(G), \text{ and }
\tau^{(n)}{G}:=(\tau^{(1)}{G};(\tau^{(n-1)}{S})_{S\in\mathrm{Lyr}_1{G}}) \text{ for } n\ge 2.
\end{equation}
\end{definition}


\subsubsection{Higher abelian type invariants of a number field}
\label{sss:TauFld}
\noindent
Within the frame of algebraic number theory,
abelian type invariants of \textit{higher order}
are defined recursively in the following way.

Let \(F\) be an algebraic number field,
denote by \(\mathrm{Cl}_p{F}\) the \(p\)-class group of \(F\),
and by \(F_p^{(1)}\) the first Hilbert \(p\)-class field of \(F\),
that is,
the maximal abelian unramified \(p\)-extension of \(F\).


\begin{definition}
\label{dfn:TauFld}
The set of all unramified cyclic extensions \(E/F\) of degree \(p\)
which are contained in the \(p\)-class field,
\begin{equation}
\label{eqn:LyrFld}
\mathrm{Lyr}_1{F}:=\lbrace E>F\mid E\le F_p^{(1)},\ \lbrack E:F\rbrack=p\rbrace
\end{equation}
is called the \textit{first layer} of extension fields of \(F\).
For any positive integer \(n\ge 1\),
\textit{abelian type invariants of \(n\)th order} of \(F\)
are defined recursively by 
\begin{equation}
\label{eqn:TauFld}
\tau^{(1)}{F}:=\mathrm{ATI}(\mathrm{Cl}_p{F}), \text{ and }
\tau^{(n)}{F}:=(\tau^{(1)}{F};(\tau^{(n-1)}{E})_{E\in\mathrm{Lyr}_1{F}}) \text{ for } n\ge 2.
\end{equation}
\end{definition}


\subsection{Transfer kernel type}
\label{ss:TKT}
\noindent
Next, we explain the concept of transfer kernel type
of finite \(p\)-groups and infinite pro-\(p\) groups.

\subsubsection{Transfer kernel type of a pro-\(p\) group}
\label{sss:TKTGrp}

\noindent
Denote by \(p\ge 2\) a prime number.
Let \(G\) be a pro-\(p\) group
with commutator subgroup \(G^\prime\)
and \textit{finite} abelianization \(G^{ab}=G/G^\prime\).

\begin{definition}
\label{dfn:TKTGrp}
By the \textit{transfer kernel type} of \(G\),
we understand the finite family
\begin{equation}
\label{eqn:TKTGrp}
\varkappa(G):=(\ker(T_{G,S}))_{S\in\mathrm{Lyr}_1{G}},
\end{equation}
where \(T_{G,S}:\,G/G^\prime\to S/S^\prime\) denotes
the transfer homomorphism from \(G\) to the normal subgroup \(S\) of finite index \((G:S)=p\).
\end{definition}


\noindent
More specifically,
suppose that \(G^{ab}\simeq C_p\times C_p\) is elementary abelian of rank two.
Then \(\mathrm{Lyr}_1{G}\) has \(p+1\) elements \(S_1,\ldots,S_{p+1}\),
the transfer kernel type of \(G\) is described briefly
by a family of non-negative integers
\(\varkappa(G)=(\varkappa_i)_{1\le i\le p+1}\in\lbrack 0,p+1\rbrack^{p+1}\)
such that
\begin{equation}
\label{eqn:TKT11Grp}
\varkappa_i:=
\begin{cases}
0 & \text{ if } \ker(T_{G,S_i})=G/G^\prime, \\
j & \text{ if } \ker(T_{G,S_i})=S_j/G^\prime \text{ for some } 1\le j\le p+1,
\end{cases}
\end{equation}
and the symmetric group \(\mathrm{S}_{p+1}\) of degree \(p+1\) acts on \(\lbrack 0,p+1\rbrack^{p+1}\)
via \(\varkappa\mapsto\varkappa^\pi:=\pi_0^{-1}\circ\varkappa\circ\pi\), for each \(\pi\in\mathrm{S}_{p+1}\),
where the extension \(\pi_0\) of \(\pi\) to \(\lbrack 0,p+1\rbrack\) fixes the zero.


\begin{definition}
\label{dfn:TKT11Grp}
The orbit \(\varkappa(G)^{\mathrm{S}_{p+1}}\) is called the \textit{invariant type} of \(G\),
but it is actually given by one of the orbit representatives \((\varkappa_i)_{1\le i\le p+1}\).
Any two distinct orbit representatives \(\lambda_1,\lambda_2\in\varkappa(G)^{\mathrm{S}_{p+1}}\)
are called \textit{equivalent}, denoted by the symbol \(\lambda_1\sim\lambda_2\).
\end{definition}


\subsubsection{Transfer kernel type of a number field}
\label{sss:TKTFld}
\noindent
Let \(F\) be an algebraic number field,
and denote by \(\mathrm{Cl}_p{F}\) the \(p\)-class group of \(F\).

\begin{definition}
\label{dfn:TKTFld}
By the \textit{transfer kernel type} of \(F\),
we understand the finite family
\begin{equation}
\label{eqn:TKTFld}
\varkappa(F):=(\ker(T_{F,E}))_{E\in\mathrm{Lyr}_1{F}},
\end{equation}
where \(T_{F,E}:\,\mathrm{Cl}_p{F}\to\mathrm{Cl}_p{E}\) denotes
the transfer of \(p\)-classes from \(F\) to the unramified cyclic extension \(E\) of degree \(\lbrack E:F\rbrack=p\),
which is also known as the \(p\)-class extension homomorphism.
\end{definition}


\noindent
More specifically,
suppose that \(\mathrm{Cl}_p{F}\simeq C_p\times C_p\) is elementary abelian of rank two.
Then \(\mathrm{Lyr}_1{F}\) has \(p+1\) elements \(E_1,\ldots,E_{p+1}\),
the transfer kernel type of \(F\) is described briefly
by a family of non-negative integers
\(\varkappa(F)=(\varkappa_i)_{1\le i\le p+1}\in\lbrack 0,p+1\rbrack^{p+1}\)
such that
\begin{equation}
\label{eqn:TKT11Fld}
\varkappa_i:=
\begin{cases}
0 & \text{ if } \ker(T_{F,E_i})=\mathrm{Cl}_p{F}, \\
j & \text{ if } \ker(T_{F,E_i})=\mathrm{Norm}_{E_j/F}(\mathrm{Cl}_p{E_j}) \text{ for some } 1\le j\le p+1,
\end{cases}
\end{equation}
and the symmetric group \(\mathrm{S}_{p+1}\) of degree \(p+1\) acts on \(\lbrack 0,p+1\rbrack^{p+1}\)
via \(\varkappa\mapsto\varkappa^\pi:=\pi_0^{-1}\circ\varkappa\circ\pi\), for each \(\pi\in\mathrm{S}_{p+1}\),
where the extension \(\pi_0\) of \(\pi\) to \(\lbrack 0,p+1\rbrack\) fixes the zero.


\begin{definition}
\label{dfn:TKT11Fld}
The orbit \(\varkappa(F)^{\mathrm{S}_{p+1}}\) is called the \textit{invariant type} of \(F\),
but it is actually given by one of the orbit representatives \((\varkappa_i)_{1\le i\le p+1}\).
Any two distinct orbit representatives \(\lambda_1,\lambda_2\in\varkappa(F)^{\mathrm{S}_{p+1}}\)
are called \textit{equivalent}, denoted by the symbol \(\lambda_1\sim\lambda_2\).
\end{definition}



\section{The Artin limit pattern}
\label{s:LimitPattern}
\noindent
Let \(p\) be a prime number.
For the recursive construction of the Artin limit pattern
of a pro-\(p\) group \(G\)
with commutator subgroup \(G^\prime\)
and \textit{finite} abelianization \(G^{ab}=G/G^\prime\),
we need the following considerations.


\subsection{Mappings of the Artin limit pattern}
\label{ss:LimitPatternMaps}
Due to our assumptions, the first layer \(\mathrm{Lyr}_1{G}\) of subgroups of \(G\) is a finite set
consisting of maximal normal subgroups \(S\) of \(G\)
with abelian quotients \(G/S\).
Consequently, the \textit{Artin transfer homomorphism}
from \(G\) to \(S\in\mathrm{Lyr}_1{G}\)
is distinguished by a very simple mapping law:
\begin{equation}
\label{eqn:Transfer}
T_{G,S}:\,G/G^\prime\to S/S^\prime,\ g\cdot G^\prime\mapsto
\begin{cases}
g^p\cdot S^\prime                        & \text{ if } g\in G/G^\prime\setminus S/G^\prime, \\
g^{1+h+h^2+\ldots+h^{p-1}}\cdot S^\prime & \text{ if } g\in S/G^\prime,
\end{cases}
\end{equation}
where \(h\) denotes an arbitrary element in \(G\setminus S\)
\cite[\S\ 4.1, p. 76]{Ma9}.

The Artin limit pattern encapsulates particular group theoretic information
(connected with Artin transfers)
about the lattice of subgroups of \(G\),
where each element \(U\) has at least one predecessor,
except the root \(G\) itself.
We select a unique predecessor in the following way:
for \(U\in\mathrm{Lyr}_1{S}\) we put \(\pi(U):=S\),
and we add the formal definition \(\pi(G):=G\).
This enables a recursive construction, as follows:

\begin{definition}
\label{dfn:Collection}
The \textit{collection of Artin transfers up to order} \(n\) of \(G\) is defined recursively by
\begin{equation}
\label{eqn:MapRecursion}
\alpha^{(1)}{G}:=T_{\pi(G),G}, \text{ and }
\alpha^{(n)}{G}:=\lbrack\alpha^{(1)}{G};(\alpha^{(n-1)}{S})_{S\in\mathrm{Lyr}_1{G}}\rbrack \text{ for } n\ge 2.
\end{equation}
The limit of this infinite recursive nesting process
is denoted by
\begin{equation}
\label{eqn:MapLimit}
\alpha^{(\infty)}{G}:=\lim_{n\to\infty}\,\alpha^{(n)}{G}
\end{equation}
and is called the \textit{Artin transfer collection} of \(G\).
\end{definition}


\begin{remark}
\label{rmk:Collection}
By means of the collection of Artin transfers up to order three,
\[\alpha^{(3)}{G}=\lbrack T_{G,G};(\alpha^{(2)}{S})_{S\in\mathrm{Lyr}_1{G}}\rbrack
=\lbrack T_{G,G};(\lbrack T_{G,S};(T_{S,U})_{U\in\mathrm{Lyr}_1{S}}\rbrack)_{S\in\mathrm{Lyr}_1{G}}\rbrack,\]
it should be emphasized that our definition of stepwise relative mappings \(T_{G,S}\) and \(T_{S,U}\)
admits finer information than the corresponding absolute mappings \(T_{G,U}=T_{S,U}\circ T_{G,S}\)
\cite[Thm. 3.3, p. 72]{Ma9},
since in general the kernel of \(T_{S,U}\) cannot be reconstructed from \(T_{G,U}\) and \(T_{G,S}\).
\end{remark}


\subsection{Objects of the Artin limit pattern}
\label{ss:LimitPatternObjects}
The infinite collection of mappings \(\alpha^{(\infty)}{G}\)
is only the foundation for the objects \(\tau^{(\infty)}{G}\) and \(\varkappa^{(\infty)}{G}\)
we are really interested in.

\begin{definition}
\label{dfn:Limit}
The \textit{iterated abelian quotient invariants up to order} \(n\) of \(G\) are defined recursively by
\begin{equation}
\label{eqn:TargetRecursion}
\tau^{(1)}{G}:=\mathrm{AQI}(G), \text{ and }
\tau^{(n)}{G}:=\lbrack\tau^{(1)}{G};(\tau^{(n-1)}{S})_{S\in\mathrm{Lyr}_1{G}}\rbrack \text{ for } n\ge 2.
\end{equation}
Similarly, the \textit{iterated transfer kernels up to order} \(n\) of \(G\) are defined recursively by
\begin{equation}
\label{eqn:KernelRecursion}
\varkappa^{(1)}{G}:=\ker(T_{\pi(G),G}), \text{ and }
\varkappa^{(n)}{G}:=\lbrack\varkappa^{(1)}{G};(\varkappa^{(n-1)}{S})_{S\in\mathrm{Lyr}_1{G}}\rbrack \text{ for } n\ge 2.
\end{equation}
Both are collected in the \(n\)th \textit{order Artin pattern} \(\mathrm{AP}^{(n)}{G}:=(\tau^{(n)}{G},\varkappa^{(n)}{G})\) of \(G\).
The limits of these infinite recursive nesting processes
are called the \textit{abelian invariant collection} of \(G\),
\begin{equation}
\label{eqn:TargetLimit}
\tau^{(\infty)}{G}:=\lim_{n\to\infty}\,\tau^{(n)}{G},
\end{equation}
and the \textit{transfer kernel collection} of \(G\),
\begin{equation}
\label{eqn:KernelLimit}
\varkappa^{(\infty)}{G}:=\lim_{n\to\infty}\,\varkappa^{(n)}{G}.
\end{equation}
Finally, the pair
\(\mathrm{ALP}(G):=(\tau^{(\infty)}{G},\varkappa^{(\infty)}{G})\)
is called the \textit{Artin limit pattern} of \(G\).
\end{definition}


\begin{remark}
\label{rmk:Limit}
For a finite \(p\)-group \(G\),
the recursive nesting processes in the definition of the Artin limit pattern
are actually finite.

The abelian quotient invariants are a \textit{unary} concept,
since \(\tau^{(1)}{G}=\mathrm{AQI}(G)=\mathrm{ATI}(G/G^\prime)\) depends on \(G\) only.
The first order abelian quotient invariants \(\tau^{(1)}{G}\) already
contain non-trivial information on the abelianization of \(G\).

The transfer kernels are a \textit{binary} concept for \(S<G\),
since \(\varkappa^{(1)}{S}=\ker(T_{\pi(S),S})\) depends on \(\pi(S)\) and \(S\).
The first order transfer kernel of \(G\) is trivial:
\(\varkappa^{(1)}{G}=\ker(T_{\pi(G),G})=\ker(T_{G,G})=\ker(\mathrm{id}_{G/G^\prime})=1\),
and non-trivial information starts with the transfer kernels of second order
\(\varkappa^{(1)}{S}=\ker(T_{\pi(S),S})=\ker(T_{G,S})\) for \(S\in\mathrm{Lyr}_1{G}\)
which are members of \(\varkappa^{(2)}{G}\).

The analogous constructions for a \textit{number field} \(F\) instead of a pro-\(p\) group \(G\),
along the lines of \S\S\
\ref{sss:TauFld}
and
\ref{sss:TKTFld},
lead to the \textit{Artin limit pattern} \(\mathrm{ALP}(F):=(\tau^{(\infty)}{F},\varkappa^{(\infty)}{F})\) of \(F\).
\end{remark}


\subsection{Connection between pro-\(p\) groups and number fields}
\label{ss:GrpFld}
\noindent
Let \(F_p^{(\infty)}\) be the Hilbert \(p\)-class tower of the number field \(F\),
that is, the maximal unramified pro-\(p\) extension of \(F\),
and denote by \(\mathrm{G}_p^\infty{F}=\mathrm{Gal}(F_p^{(\infty)}/F)\)
its Galois group, which is briefly called the \textit{\(p\)-tower group} of \(F\).
Now we are going to employ the
abelian type invariant collection \(\tau^{(\infty)}{F}\) of \(F\),
and the abelian quotient invariant collection
\(\tau^{(\infty)}(\mathrm{G}_p^\infty{F})\) of \(\mathrm{G}_p^\infty{F}\),
i.e., the first component of the respective Artin limit pattern.
The transfer kernel collections \(\varkappa^{(\infty)}\) will be considered further in \S\
\ref{s:MaximalSubgroups}.

\begin{theorem}
\label{thm:TauGrpFld}
For each integer \(n\ge 1\),
the abelian quotient invariants of \(n\)th order of the \(p\)-tower group \(\mathrm{G}_p^\infty{F}\) of \(F\)
are equal to the abelian type invariants of \(n\)th order of the number field \(F\)
\begin{equation}
\label{eqn:TauGrpFld}
(\forall n\ge 1) \quad \tau^{(n)}(\mathrm{G}_p^\infty{F})=\tau^{(n)}{F}, \text{ and thus } \tau^{(\infty)}(\mathrm{G}_p^\infty{F})=\tau^{(\infty)}{F}.
\end{equation}
The invariant type of the \(p\)-tower group \(\mathrm{G}_p^\infty{F}\) of \(F\)
coincides with the invariant type of the number field \(F\)
\begin{equation}
\label{eqn:TKTGrpFld}
\varkappa(\mathrm{G}_p^\infty{F})^{S_{p+1}}=\varkappa(F)^{S_{p+1}}.
\end{equation}
Even the orbit representatives of the transfer kernel types of \(\mathrm{G}_p^\infty{F}\) and \(F\) coincide,
\begin{equation}
\label{eqn:TKT11GrpFld}
\varkappa(\mathrm{G}_p^\infty{F})=(\ker(T_{\mathrm{G}_p^\infty{F},U_i}))_{1\le i\le p+1}=
(\ker(T_{F,E_i}))_{1\le i\le p+1}=\varkappa(F),
\end{equation}
provided that the \(U_i\in\mathrm{Lyr}_1(\mathrm{G}_p^\infty{F})\) and the \(E_i\in\mathrm{Lyr}_1{F}\)
are connected by
\(U_i=\mathrm{Gal}(F_p^{(\infty)}/E_i)\), for each \(1\le i\le p+1\).
Otherwise, we only have equivalence
\(\varkappa(\mathrm{G}_p^\infty{F})\sim\varkappa(F)\).
\end{theorem}

\begin{proof}
The claims are well-known consequences of the Artin reciprocity law of class field theory
\cite{Ar1,Ar2}.
\end{proof}


In contrast to the full \(p\)-tower group \(\mathfrak{G}=\mathrm{G}_p^\infty{F}\),
the Galois groups \(\mathrm{G}_p^m{F}:=\mathrm{Gal}(F_p^{(m)}/F)\simeq\mathfrak{G}/\mathfrak{G}^{(m)}\)
of the finite stages \(F_p^{(m)}\) of the \(p\)-class tower of \(F\),
that is, of the higher Hilbert \(p\)-class fields of the number field \(F\),
in general fail to reveal the abelian type invariants of \(n\)th order of the number field \(F\).
More precisely, there is a strict upper bound on the order \(n\) of the ATI of \(F\)
which coincide with the AQI of order \(n\) of the \(m\)th \(p\)-class group \(\mathrm{G}_p^m{F}\) of \(F\)
with a fixed integer \(m\ge 0\), namely the bound \(n\le m\).

\begin{theorem}
\label{thm:SuccessiveApproximation}
\textbf{(Successive Approximation Theorem.)}\\
Let \(F\) be a number field, \(p\) a prime, and \(m,n\) integers.
The abelian invariant collection \(\tau^{(\infty)}{F}\) of \(F\) is approximated successively
by the iterated AQI of sufficiently high \(p\)-class groups of \(F\):
\begin{equation}
\label{eqn:TauSuccessiveApproximation}
(\forall n\ge 1) \quad (\forall m\ge n) \quad \tau^{(n)}(\mathrm{G}_p^m{F})=\tau^{(n)}{F}.
\end{equation}
However, the transfer kernel type is a phenomenon of second order:
\begin{equation}
\label{eqn:TKTSuccessiveApproximation}
(\forall m\ge 2) \quad \varkappa(\mathrm{G}_p^m{F})\sim\varkappa(F),
\end{equation}
in particular, the metabelian second \(p\)-class group
\(\mathfrak{M}:=\mathrm{G}_p^2{F}\simeq\mathfrak{G}/\mathfrak{G}^{\prime\prime}\) of \(F\)
is sufficient for determining the transfer kernel type of \(F\).
\end{theorem}

\begin{proof}
This is one of the main results of
\cite[Thm. 1.19, p. 78]{Ma15}
and
\cite[p. 13]{Ma15b}.
\end{proof}

In general, the upper bound on the order \(n\) of the ATI of \(F\) in Theorem
\ref{thm:SuccessiveApproximation}
seems to be sharp, in the following sense,
where \(m=n-1\).

\begin{conjecture}
\label{cnj:StageSeparation}
\textbf{(Stage Separation Criterion.)}\\
Denote by \(\ell_p{F}\) the length of the \(p\)-class tower of \(F\),
that is the derived length \(\mathrm{dl}(\mathrm{G}_p^\infty{F})\) of the \(p\)-tower group of \(F\).
It is determined in terms of iterated AQI of higher \(p\)-class groups of \(F\) by the following condition: 
\begin{equation}
\label{eqn:TauStageSeparation}
(\forall n\ge 1) \quad \lbrack\ \ell_p{F}\ge n\ \Longleftrightarrow\ \tau^{(n)}(\mathrm{G}_p^{n-1}{F})<\tau^{(n)}{F}\ \rbrack.
\end{equation}
\end{conjecture}

\noindent
The sufficiency of the condition in Conjecture
\ref{cnj:StageSeparation}
is a proven theorem
\cite[p. 13]{Ma15b}.



\section{Successive approximation of the \(p\)-class tower}
\label{s:Approximation}


\subsection{Computational perspectives}
\label{ss:Computation}
\noindent
Our first attempt to find sound asymptotic tendencies
in the distribution of higher non-abelian \(p\)-class groups
\(\mathrm{G}_p^n{F}=\mathrm{Gal}(F_p^{(n)}/F)\), with \(n\ge 2\),
among the finite \(p\)-groups
was planned in \(1991\) already
\cite[\S\ 3, Remark, p. 77]{Ma}.
However, the insurmountable obstacles in the required computations
limited the progress for twenty years.
In \(2012\), we finally succeeded in the significant break-through
of computing the second \(3\)-class groups
\(\mathfrak{M}=\mathrm{G}_3^2{F}\),
that is, the metabelianizations \(\mathfrak{G}/\mathfrak{G}^{(2)}\)
of the \(3\)-class tower groups \(\mathfrak{G}=\mathrm{Gal}(F_3^{(\infty)}/F)\)
of all \(4596\) quadratic fields \(F=\mathbb{Q}(\sqrt{d})\)
with fundamental discriminants in the remarkable range \(-10^6<d<10^7\)
and elementary bicyclic \(3\)-class group \(\mathrm{Cl}_3{F}\simeq C_3\times C_3\) of rank two
\cite[\S\ 6, pp. 495--499]{Ma1}.
The underlying computational techniques were based on the
\textit{principalization algorithm via class group structure}
which we had invented in \(2009\) and implemented by means of
the number theoretic computer algebra system PARI/GP
\cite{PARI}
in \(2010\), as described in
\cite[\S\S\ 5--6, pp. 446--455]{Ma3}.

Throughout this paper,
isomorphism classes of finite groups \(G\) are characterized uniquely
by their identifier in the SmallGroups Database
\cite{BEO1,BEO2},
which is denoted by a pair \(\langle o,i\rangle\)
consisting of the order \(o=\mathrm{ord}(G)\) and a positive integer \(i\),
delimited with angle brackets.
The counter \(1\le i\le N(o)\) is unique for a fixed value of the order \(o\).
In the computational algebra system MAGMA
\cite{BCP,BCFS,MAGMA},
the upper bound \(N(o)\) can be obtained as return value of the function
\texttt{NumberOfSmallGroups}\((o)\),
provided that
\texttt{IsInSmallGroupDatabase}\((o)\)
returns \texttt{true}.
The identifier of a given finite group \(G\)
can be retrieved as return value of the function
\texttt{IdentifyGroup}\((G)\),
provided that
\texttt{CanIdentifyGroup}\((o)\)
returns \texttt{true}.



\subsection{Trivial towers with \(\ell_p{F}=0\)}
\label{ss:Trivial}
\noindent
For the decision if the \(p\)-class tower of a number field \(F\) is trivial
with length \(\ell_p{F}=0\)
it suffices to compute the class number \(h(F)\) of the field.

\begin{theorem}
\label{thm:Trivial}
\textbf{(Trivial \(p\)-class tower.)}\\
The \(p\)-class tower of a number field \(F\) is trivial, \(F_p^{(\infty)}=F\),
with length \(\ell_p{F}=0\),
if and only if the class number \(h(F)=\#\mathrm{Cl}(F)\) is not divisible by \(p\),
i. e., the \(p\)-class number is \(h_p{F}=1\).
\end{theorem}

\begin{proof}
The proof consists of a sequence of equivalent statements:
The class number satisfies \(p\nmid h(F)\). \(\Longleftrightarrow\)
The \(p\)-valuation of \(h(F)\) is \(v_p(h(F))=0\). \(\Longleftrightarrow\)
The \(p\)-class number is \(\#\mathrm{Cl}_p{F}=h_p{F}=p^{v_p(h(F))}=1\). \(\Longleftrightarrow\)
The \(p\)-class group \(\mathrm{Cl}_p{F}=1\) is trivial. \(\Longleftrightarrow\)
The \(p\)-class rank \(\rho_p=\dim_{\mathbb{F}_p}(\mathrm{Cl}(F)/\mathrm{Cl}(F)^p)\) is equal to zero. \(\Longleftrightarrow\)
The number of unramified cyclic extensions \(E/F\) of degree \(p\) is \(\frac{p^{\rho_p}-1}{p-1}=\frac{p^0-1}{p-1}=\frac{1-1}{p-1}=0\). \(\Longleftrightarrow\)
The maximal unramified \(p\)-extension \(F_p^{(\infty)}\) of \(F\) coincides with \(F\). \(\Longleftrightarrow\)
The Galois group \(\mathrm{G}_p^\infty{F}=\mathrm{Gal}(F_p^{(\infty)}/F)=\mathrm{Gal}(F/F)=1\) is trivial. \(\Longleftrightarrow\)
The length of the \(p\)-class tower is \(\ell_p{F}=\mathrm{dl}(\mathrm{G}_p^\infty{F})=\mathrm{dl}(1)=0\).
\end{proof}


\noindent
Already C. F. Gauss was able to compute class numbers \(h(F)\) of quadratic fields \(F=\mathbb{Q}(\sqrt{d})\),
at a time when the concept of class field theory was not yet coined.
Nowadays, there exist extensive tables of quadratic class numbers
which even contain the structures of the associated class groups \(\mathrm{Cl}(F)\).
In \(1998\), Jacobson
\cite{Js}
covered all real quadratic fields
with positive discriminants in the range \(0<d<10^9\),
and in \(2016\), Mosunov and Jacobson
\cite{MsJs}
investigated all imaginary quadratic fields
with negative discriminants \(-10^{12}<d<0\).
Now we apply these results to class field theory.

\begin{corollary}
\label{cor:ImaginaryTrivial}
\textbf{(Statistics for \(p=3\).)}
The asymptotic proportion of imaginary quadratic fields \(F=\mathbb{Q}(\sqrt{d})\),
with negative discriminants \(d<0\),
whose class number \(h(F)\) is, respectively is not, divisible by \(p=3\)
is given as \(43.99\%\), respectively \(56.01\%\), by the heuristics of Cohen, Lenstra and Martinet.
In Table
\ref{tbl:ImaginaryTrivial},
the approximations of these theoretical limits
by relative frequencies in various ranges \(L<d<0\)
are shown.
\end{corollary}


\renewcommand{\arraystretch}{1.2}

\begin{table}[ht]
\caption{Imaginary quadratic fields \(F\) with non-trivial, resp. trivial, \(3\)-class tower}
\label{tbl:ImaginaryTrivial}
\begin{center}
\begin{tabular}{|c||r|c||r|c||c|}
\hline
 \(L\)        & \(\#(3\mid h(F))\)     & rel. fr.    & \(\#(3\nmid h(F))\)    & rel. fr.    & w. r. t. \(\#\)total   \\
\hline
 \(-10^6\)    &           \(121\,645\) & \(40.02\%\) &           \(182\,323\) & \(59.98\%\) &           \(303\,968\) \\
 \(-10^{11}\) &  \(13\,206\,088\,529\) & \(43.45\%\) &  \(17\,190\,266\,523\) & \(56.55\%\) &  \(30\,396\,355\,052\) \\
 \(-10^{12}\) & \(132\,584\,350\,621\) & \(43.62\%\) & \(171\,379\,200\,091\) & \(56.38\%\) & \(303\,963\,550\,712\) \\
\hline
\end{tabular}
\end{center}
\end{table}


\begin{proof}
The heuristic asymptotic limits are given in
\cite[\S\ 2, (1.1.c), p. 126]{ChMt}.
Their approximation by discriminants \(L<d<0\) with \(L=-10^6\) in
\cite[Example, p. 843]{Ma0}
and
\cite[\S\ 2, Remark, and \S\ 3, Remark, p. 77]{Ma},
where \(118\,455+3\,190=121\,645\),
is still rather far away from the limits.
In contrast, the approximations associated with the bounds \(L=-10^{11}\) and \(L=-10^{12}\) in
\cite[p. 2001]{MsJs}
are very close already.
\end{proof}



\subsection{Abelian single-stage towers with \(\ell_p{F}=1\)}
\label{ss:SingleStage}
\noindent
The first stage of the \(p\)-class tower of a number field \(F\)
is determined by the structure of the \(p\)-class group \(\mathrm{Cl}_p{F}\) of \(F\)
as a finite abelian \(p\)-group.
This is exactly the first order Artin pattern
\begin{equation}
\label{eqn:AP1}
\mathrm{AP}^{(1)}{F}=(\tau^{(1)}{F},\varkappa^{(1)}{F})=(\mathrm{ATI}(\mathrm{Cl}_p{F}),\ker(T_{F,F})),
\end{equation}
since the trivial \(\ker(T_{F,F})=1\) does not contain information.
However, only in the case of \(p\)-class rank one,
\(\rho_p=\dim_{\mathbb{F}_p}(\mathrm{Cl}(F)/\mathrm{Cl}(F)^p)=1\),
it is warranted that the exact length of the tower is \(\ell_p{F}=1\).
A statistical example
\cite[\S\ 2, Remark, p. 77]{Ma}
is shown in Table
\ref{tbl:ImaginaryCyclic}. 

\begin{theorem}
\label{thm:Abelian}
A number field \(F\) with non-trivial cyclic \(p\)-class group \(\mathrm{Cl}_p{F}\)
has an abelian \(p\)-class tower of exact length \(\ell_p{F}=1\),
in fact, the Galois group \(\mathrm{G}_p^\infty{F}\simeq\mathrm{G}_p^1{F}\simeq\mathrm{Cl}_p{F}\) is cyclic.
\end{theorem}

\begin{proof}
Suppose that \(\mathrm{Cl}_p{F}>1\) is non-trivial and cyclic.
If the \(p\)-class tower had a length \(\ell_p{F}\ge 2\),
the second \(p\)-class group \(\mathfrak{M}=\mathrm{G}_p^2{F}\)
would be a non-abelian finite \(p\)-group
with cyclic abelianization \(\mathfrak{M}/\mathfrak{M}^\prime\simeq\mathrm{Cl}_p{F}\).
However, it is well known that a nilpotent group with cyclic abelianization is abelian,
which contradicts the assumption of a length \(\ell_p{F}\ge 2\).
\end{proof}


\renewcommand{\arraystretch}{1.2}

\begin{table}[ht]
\caption{Imaginary quadratic fields \(F\) with cyclic \(3\)-class tower for \(-10^6<d<0\)}
\label{tbl:ImaginaryCyclic}
\begin{center}
\begin{tabular}{|c||r|c||c|}
\hline
 \(\mathrm{Cl}_3{F}\simeq\) & abs. fr.    & rel. fr.    & w. r. t. \(\#(\rho_3=1)\) \\
\hline
 \(C_3\)                    & \(80\,115\) & \(67.63\%\) &              \(118\,455\) \\
 \(C_9\)                    & \(26\,458\) & \(22.34\%\) &              \(118\,455\) \\
 \(C_{27}\)                 &  \(8\,974\) &  \(7.58\%\) &              \(118\,455\) \\
 \(C_{81}\)                 &  \(2\,472\) &  \(2.09\%\) &              \(118\,455\) \\
 \(C_{243}\)                &     \(393\) &  \(0.33\%\) &              \(118\,455\) \\
 \(C_{729}\)                &      \(43\) &  \(0.04\%\) &              \(118\,455\) \\
\hline
\end{tabular}
\end{center}
\end{table}


\begin{remark}
\label{rmk:SingleStage}
We interpret
the computation of abelian type invariants \(\tau^{(1)}{F}\) of the Sylow \(3\)-subgroup \(\mathrm{Cl}_3{F}\)
of the ideal class group \(\mathrm{Cl}(F)\) of a quadratic field \(F=\mathbb{Q}(\sqrt{d})\)
as the determination of the single-stage approximation \(\mathfrak{G}/\mathfrak{G}^\prime\simeq\mathrm{G}_3^1{F}\simeq\mathrm{Cl}_3{F}\)
of the \(3\)-class tower group \(\mathfrak{G}=\mathrm{G}_3^\infty{F}\) of \(F\).
This step yields complete information about the lattice of all
unramified abelian \(3\)-extensions \(E/F\) within the Hilbert \(3\)-class field \(\mathrm{F}_3^1{F}\) of \(F\).
\end{remark}



\subsection{Metabelian two-stage towers with \(\ell_p{F}=2\)}
\label{ss:TwoStage}
\noindent
According to the Successive Approximation Theorem
\ref{thm:SuccessiveApproximation},
the second stage \(F_p^{(2)}\) of the \(p\)-class tower of a number field \(F\)
is determined by the second order Artin pattern
\begin{equation}
\label{eqn:AP2}
\mathrm{AP}^{(2)}{F}=(\tau^{(2)}{F},\varkappa^{(2)}{F})=
(\lbrack\mathrm{ATI}(\mathrm{Cl}_p{F});(\mathrm{ATI}(\mathrm{Cl}_p{E}))_{E\in\mathrm{Lyr}_1{F}}\rbrack,\lbrack\ker(T_{F,F});(\ker(T_{F,E}))_{E\in\mathrm{Lyr}_1{F}}\rbrack).
\end{equation}

The determination of \(\mathrm{AP}^{(2)}{F}\) for a quadratic field \(F\) with \(3\)-class rank \(\rho_3=2\)
requires the computation of four \(3\)-class groups \(\mathrm{Cl}_3{E_i}\) of unramified cyclic cubic extensions \(E_1,\ldots,E_4\)
and of four transfer kernels \(\ker(T_{F,E_i})\).

Whereas Mosunov and Jacobson
\cite{MsJs}
were able to determine the class groups \(\mathrm{Cl}(F)\) of more than \(300\) billion,
precisely \(303\,963\,550\,712\),
imaginary quadratic fields \(F\) with discriminants \(-10^{12}<d<0\)
by parallel processes on multiple cores of a supercomputer
in several years of total CPU time,
it is currently definitely out of scope to compute the class groups \(\mathrm{Cl}(E_i)\), \(1\le i\le 4\),
for the \(22\,757\,307\,168\) unramified cyclic cubic extensions \(E_i/F\), of absolute degree six,
of the \(5\,689\,326\,792\) imaginary quadratic fields \(F\) with discriminants \(-10^{12}<d<0\)
and \(3\)-class rank \(\rho_3=2\).

Therefore, it must not be underestimated that Boston, Bush and Hajir
\cite{BBH}
succeeded in completing this task for the smaller range \(-10^8<d<0\)
with \(461\,925\) imaginary quadratic fields \(F\) having \(3\)-class rank \(\rho_3=2\),
and \(1\,847\,700\) associated \textit{totally complex dihedral fields} \(E_i\) of degree six
\cite[Prp. 4.1, p. 482]{Ma1}.
For this purpose the authors used the computational algebra system MAGMA
\cite{BCP,BCFS,MAGMA}
in a distributed process involving several processors with multiple cores.
\(276\,375\) of these quadratic fields \(F\)
have a \(3\)-class group \(\mathrm{Cl}_3{F}\simeq C_3\times C_3\).

Imaginary quadratic fields \(F=\mathbb{Q}(\sqrt{d})\) with negative discriminants \(d<0\) 
are the simplest number fields with respect to their unit group \(U_F\),
which is a finite torsion group of Dirichlet unit rank zero.
This fact has considerable consequences for their \(p\)-class tower groups,
according to the Shafarevich theorem
\cite{Sh},
corrected in
\cite[Thm. 5.1, p. 28]{Ma10},
\cite{Ma10a}.

\begin{theorem}
\label{thm:ImaginaryTwoStage}
Among the finite \(3\)-groups \(G\)
with elementary bicyclic abelianization \(G/G^\prime\simeq C_3\times C_3\) of rank two,
there exist only two metabelian groups with GI-action and relation rank \(d_2{G}=2\)
(so-called Schur \(\sigma\)-groups
\cite{KoVe,BBH}),
namely \(\langle 243,5\rangle\) and \(\langle 243,7\rangle\).
\begin{enumerate}
\item
These are the groups of smallest order
which are admissible as \(3\)-class tower groups \(G\simeq\mathrm{G}_3^\infty{F}\)
of imaginary quadratic fields \(F\) with \(3\)-class group \(\mathrm{Cl}_3{F}\simeq C_3\times C_3\).
\item
Generally, for any number field \(F\), these groups are determined uniquely by the second order Artin pattern.
\begin{enumerate}
\item
If \(\mathrm{AP}^{(2)}{F}=(\lbrack 1^2;(21,21,1^3,21)\rbrack,\lbrack 1;(2241)\rbrack)\)
then \(\mathrm{G}_3^\infty{F}\simeq\langle 243,5\rangle\).
\item
If \(\mathrm{AP}^{(2)}{F}=(\lbrack 1^2;(1^3,21,1^3,21)\rbrack,\lbrack 1;(4224)\rbrack)\)
then \(\mathrm{G}_3^\infty{F}\simeq\langle 243,7\rangle\).
\end{enumerate}
\item
The actual distribution of these \(3\)-class tower groups \(G\)
among the \(276\,375\) imaginary quadratic fields \(F=\mathbb{Q}(\sqrt{d})\)
with \(3\)-class group \(\mathrm{Cl}_3{F}\simeq C_3\times C_3\) and
discriminants \(-10^8<d<0\) is presented in
Table
\ref{tbl:ImaginaryTwoStage}.
\end{enumerate}
\end{theorem}


\renewcommand{\arraystretch}{1.2}

\begin{table}[ht]
\caption{Frequencies of metabelian \(3\)-class tower groups \(G\) for \(-10^8<d<0\)}
\label{tbl:ImaginaryTwoStage}
\begin{center}
\begin{tabular}{|c|r||r|c||r|c||c|r|}
\hline
 \(G\simeq\)              & abs. fr.    & rel. fr.    & w. r. t.     & rel. fr.    & w. r. t.      & measure \cite{BBH}         & \(\lvert d\rvert_{\text{min}}\) \\
\hline
 \(\langle 243,5\rangle\) & \(83\,353\) & \(30.16\%\) & \(276\,375\) & \(18.04\%\) & \(461\,925\)  & \(128/729\approx 17.56\%\) &  \(4\,027\)                     \\
 \(\langle 243,7\rangle\) & \(41\,398\) & \(14.98\%\) & \(276\,375\) &  \(8.96\%\) & \(461\,925\)  & \(64/729\approx 8.78\%\)   & \(12\,131\)                     \\
\hline
\end{tabular}
\end{center}
\end{table}


\begin{proof}
All finite \(3\)-groups \(G\) with abelianization \(G/G^\prime\simeq C_3\times C_3\)
are vertices of the descendant tree \(\mathcal{T}(R)\)
with abelian root \(R=\langle 9,2\rangle\simeq C_3\times C_3\).
A search for metabelian vertices with relation rank \(d_2{G}=2\)
in this tree yields three hits
\(\langle 27,4\rangle\), \(\langle 243,5\rangle\), and \(\langle 243,7\rangle\),
but only the latter two of them possess a GI-action.

The abelianization \(G/G^\prime\) of a finite \(3\)-group \(G\)
which is realized as the \(3\)-class tower group \(\mathrm{G}_p^\infty{F}\)
of an algebraic number field \(F\)
is isomorphic to the \(3\)-class group \(\mathrm{Cl}_3{F}\) of \(F\).
When \(F\) is imaginary quadratic, it possesses signature \((r_1,r_2)=(0,1)\)
and torsionfree Dirichlet unit rank \(r=r_1+r_2-1=0\).
If \(G/G^\prime\simeq\mathrm{Cl}_3{F}\simeq C_3\times C_3\),
then the generator rank of \(G\) is \(d_1{G}=2\)
and the Shafarevich theorem implies bounds for the relation rank
\(2=d_1{G}\le d_2{G}\le d_1{G}+r=2\).

The entries of Table
\ref{tbl:ImaginaryTwoStage}
have been taken from
\cite{BBH}.
\end{proof}


More recently, Boston, Bush and Hajir
\cite{BBH2}
used  MAGMA
\cite{MAGMA}
for computing the class groups of the \(481\,756\) real quadratic fields \(F\)
having \(3\)-class rank \(\rho_3=2\) and discriminants in the range \(0<d<10^9\),
and the class groups of the \(1\,927\,024\) associated \textit{totally real dihedral fields} \(E_i\) of degree six,
arising from unramified cyclic cubic extensions \(E_i/F\)
\cite[Prp. 4.1, p. 482]{Ma1}.
\(415\,698\) of these quadratic fields \(F\)
have a \(3\)-class group \(\mathrm{Cl}_3{F}\simeq C_3\times C_3\)
(\(415\,699\) according to
\cite[Tbl. 7]{Js}).

Real quadratic fields \(F=\mathbb{Q}(\sqrt{d})\) with positive discriminants \(d>0\) 
are the second simplest number fields with respect to their unit group \(U_F\),
which is an infinite group of torsionfree Dirichlet unit rank one.
Again, there are remarkable consequences for their \(p\)-tower groups,
by the Shafarevich theorem
\cite[Thm. 5.1, p. 28]{Ma10}.

\begin{theorem}
\label{thm:RealTwoStage}
Among the finite \(3\)-groups \(G\)
with elementary bicyclic abelianization \(G/G^\prime\simeq C_3\times C_3\) of rank two,
there exist infinitely many metabelian groups with GI-action and relation rank \(d_2{G}=3\)
(so-called Schur\(+1\) \(\sigma\)-groups
\cite{BBH2}),
but only three of minimal order \(3^4\),
namely \(\langle 81,7\rangle\),\(\langle 81,8\rangle\) and \(\langle 81,10\rangle\).
\begin{enumerate}
\item
These are the groups of smallest order
which are admissible as \(3\)-class tower groups \(G\simeq\mathrm{G}_3^\infty{F}\)
of real quadratic fields \(F\) with \(3\)-class group \(\mathrm{Cl}_3{F}\simeq C_3\times C_3\).
\item
Generally, for any number field \(F\), these groups are determined uniquely by the second order Artin pattern.
\begin{enumerate}
\item
If \(\mathrm{AP}^{(2)}{F}=(\lbrack 1^2;(1^3,1^2,1^2,1^2)\rbrack,\lbrack 1;(2000)\rbrack)\)
then \(\mathrm{G}_3^\infty{F}\simeq\langle 81,7\rangle\).
\item
If \(\mathrm{AP}^{(2)}{F}=(\lbrack 1^2;(21,1^2,1^2,1^2)\rbrack,\lbrack 1;(2000)\rbrack)\)
then \(\mathrm{G}_3^\infty{F}\simeq\langle 81,8\rangle\).
\item
If \(\mathrm{AP}^{(2)}{F}=(\lbrack 1^2;(21,1^2,1^2,1^2)\rbrack,\lbrack 1;(1000)\rbrack)\)
then \(\mathrm{G}_3^\infty{F}\simeq\langle 81,10\rangle\).
\end{enumerate}
\item
The actual distribution of these \(3\)-class tower groups \(G\)
among the \(415\,698\) real quadratic fields \(F=\mathbb{Q}(\sqrt{d})\)
with \(3\)-class group \(\mathrm{Cl}_3{F}\simeq C_3\times C_3\) and
discriminants \(0<d<10^9\) is presented in
Table
\ref{tbl:RealTwoStage9}.
Additionally, the frequencies of the groups \(\langle 243,5\rangle\) and \(\langle 243,7\rangle\) in Theorem
\ref{thm:ImaginaryTwoStage}
are given.
\end{enumerate}
\end{theorem}


\renewcommand{\arraystretch}{1.2}

\begin{table}[ht]
\caption{Frequencies of metabelian \(3\)-class tower groups \(G\) for \(0<d<10^9\)}
\label{tbl:RealTwoStage9}
\begin{center}
\begin{tabular}{|c|r||r|c||r|c||c|r|}
\hline
 \(G\simeq\)                & abs. fr.     & rel. fr.    & w. r. t.     & rel. fr.    & w. r. t.      & measure  \cite{BBH2}          & \(d_{\text{min}}\) \\
\hline
 \(\langle 81,7\rangle\)    & \(122\,955\) & \(29.58\%\) & \(415\,698\) & \(25.52\%\) & \(481\,756\)  & \(1664/6561\approx 25.36\%\)  & \(142\,097\)       \\
\hline
 \(\langle 81,8\rangle\) or & \(208\,236\) & \(50.09\%\) & \(415\,698\) & \(43.22\%\) & \(481\,756\)  & \(8320/19683\approx 42.27\%\) &  \(32\,009\)       \\
 \(\langle 81,10\rangle\)   &              &             &              &             &               &                               &                    \\
\hline
 \(\langle 243,5\rangle\)   &  \(13\,712\) &  \(3.30\%\) & \(415\,698\) &  \(2.85\%\) & \(481\,756\)  & \(1664/59049\approx 2.82\%\)  & \(422\,573\)       \\
\hline
 \(\langle 243,7\rangle\)   &   \(6\,691\) &  \(1.61\%\) & \(415\,698\) &  \(1.39\%\) & \(481\,756\)  & \(832/59049\approx 1.41\%\)   & \(631\,769\)       \\
\hline
\end{tabular}
\end{center}
\end{table}


\begin{proof}
A search for metabelian vertices \(G\) of minimal order with relation rank \(d_2{G}=3\)
in the descendant tree \(\mathcal{T}(R)\)
with abelian root \(R=\langle 9,2\rangle\simeq C_3\times C_3\)
yields three hits
\(\langle 27,7\rangle\), \(\langle 27,8\rangle\), and \(\langle 27,10\rangle\).
All of them possess a GI-action.

The abelianization \(G/G^\prime\) of a finite \(3\)-group \(G\)
which is realized as the \(3\)-class tower group \(\mathrm{G}_p^\infty{F}\)
of an algebraic number field \(F\)
is isomorphic to the \(3\)-class group \(\mathrm{Cl}_3{F}\) of \(F\).
When \(F\) is real quadratic, it possesses signature \((r_1,r_2)=(2,0)\)
and torsionfree Dirichlet unit rank \(r=r_1+r_2-1=1\).
If \(G/G^\prime\simeq\mathrm{Cl}_3{F}\simeq C_3\times C_3\),
then the generator rank of \(G\) is \(d_1{G}=2\)
and the Shafarevich theorem implies bounds for the relation rank
\(2=d_1{G}\le d_2{G}\le d_1{G}+r=3\).

The entries of Table
\ref{tbl:RealTwoStage9}
have been taken from
\cite{BBH2}.
\end{proof}


In
\cite{BBH2},
Boston, Bush and Hajir only computed the first component of the second order Artin pattern
\(\mathrm{AP}^{(2)}{F}=(\tau^{(2)}{F},\varkappa^{(2)}{F})\) in Formula
\eqref{eqn:AP2},
that is, the abelian type invariants \(\tau^{(2)}{F}\) of second order
of real quadratic fields \(F\) with discriminants \(0<d<10^9\).
Determining the second component \(\varkappa^{(2)}{F}\), the transfer kernel type of \(F\),
is considerably harder with respect to the computational expense.
Consequently, the most extensive numerical results on transfer kernels available currently,
have been computed by ourselves for the smaller ranges \(0<d<10^8\) in
\cite{Ma14,Ma14b},
and, even computing third order Artin patterns, for \(0<d<10^7\) in
\cite{Ma17,Ma17b}.
With the aid of these results, we now illustrate that
the transfer kernels \(\ker(T_{F,E})\) of \(3\)-class extensions
\(T_{F,E}:\mathrm{Cl}_3{F}\to\mathrm{Cl}_3{E}\)
from real quadratic fields \(F\) to unramified cyclic cubic extensions \(E/F\)
are capable of narrowing down the number of contestants
for the \(3\)-tower group \(\mathrm{G}_3^\infty{F}\) significantly,
and thus of refining the statistics in
\cite{BBH2}.


\renewcommand{\arraystretch}{1.2}

\begin{table}[ht]
\caption{Frequencies of metabelian \(3\)-class tower groups \(G\) for \(0<d<10^8\) resp. \(10^7\)}
\label{tbl:RealTwoStage8}
\begin{center}
\begin{tabular}{|c|r||r|c||r|}
\hline
 \(G\simeq\)                  & abs. fr.     & rel. fr.    & w. r. t.   & \(d_{\text{min}}\) \\
\hline
 \(\langle 81,7\rangle\)      & \(10\,244\) & \(29.58\%\) & \(34\,631\) & \(142\,097\)       \\
 \(\langle 81,8\rangle\)      & \(10\,514\) & \(30.36\%\) & \(34\,631\) &  \(32\,009\)       \\
 \(\langle 81,10\rangle\)     &  \(7\,104\) & \(20.51\%\) & \(34\,631\) &  \(72\,329\)       \\
\hline
 \(\langle 729,96\rangle\)    &     \(242\) &  \(0.70\%\) & \(34\,631\) & \(790\,085\)       \\
\hline
 \(\langle 729,97\rangle\) or &     \(713\) &  \(2.06\%\) & \(34\,631\) & \(494\,236\)       \\
 \(\langle 729,98\rangle\)    &             &             &             &                    \\
\hline
 \(\langle 729,99\rangle\)    &      \(66\) &  \(2.56\%\) &  \(2\,576\) &  \(62\,501\)       \\
 \(\langle 729,100\rangle\)   &      \(42\) &  \(1.63\%\) &  \(2\,576\) & \(152\,949\)       \\
 \(\langle 729,101\rangle\)   &      \(42\) &  \(1.63\%\) &  \(2\,576\) & \(252\,977\)       \\
\hline
\end{tabular}
\end{center}
\end{table}


\begin{corollary}
\label{cor:RealTwoStage}
\begin{enumerate}
\item
If \(\mathrm{AP}^{(2)}{F}=(\lbrack 1^2;(32,1^2,1^2,1^2)\rbrack,\lbrack 1;(1000)\rbrack)\)
then \(\mathrm{G}_3^\infty{F}\simeq\langle 729,96\rangle\).
\item
If \(\mathrm{AP}^{(2)}{F}=(\lbrack 1^2;(32,1^2,1^2,1^2)\rbrack,\lbrack 1;(2000)\rbrack)\)
then \(\mathrm{G}_3^\infty{F}\simeq\langle 729,i\rangle\) with \(i\in\lbrace 97,98\rbrace\).
\item
If \(\mathrm{AP}^{(2)}{F}=(\lbrack 1^2;(2^2,1^2,1^2,1^2)\rbrack,\lbrack 1;(0000)\rbrack)\)
then \(\mathrm{G}_3^\infty{F}\simeq\langle 729,i\rangle\) with \(i\in\lbrace 99,100,101\rbrace\).
\end{enumerate}
\item
The actual distribution of these \(3\)-class tower groups \(G\)
among the \(34\,631\), respectively \(2\,576\), real quadratic fields \(F=\mathbb{Q}(\sqrt{d})\)
with \(3\)-class group \(\mathrm{Cl}_3{F}\simeq C_3\times C_3\) and
discriminants \(0<d<10^8\), respectively \(0<d<10^7\), is presented in
Table
\ref{tbl:RealTwoStage8}.
\end{corollary}



\subsection{Non-metabelian three-stage towers with \(\ell_p{F}=3\)}
\label{ss:ThreeStage}
\noindent
According to the Successive Approximation Theorem,
the third stage \(F_p^{(3)}\) of the \(p\)-class tower of a number field \(F\)
is usually determined by the third order Artin pattern
\begin{equation}
\label{eqn:AP3}
\mathrm{AP}^{(3)}{F}=(\tau^{(3)}{F},\varkappa^{(3)}{F})=
(\lbrack\tau^{(1)}{F};(\tau^{(2)}{E})_{E\in\mathrm{Lyr}_1{F}}\rbrack,\lbrack\varkappa^{(1)}{F};(\varkappa^{(2)}{E})_{E\in\mathrm{Lyr}_1{F}}\rbrack).
\end{equation}
It is interesting, however, that there are extensive collections of quadratic fields \(F\)
with \(3\)-class towers of exact length \(\ell_3{F}=3\),
which can be characterized by the second order Artin pattern already.
We begin with imaginary quadratic fields \(F=\mathbb{Q}(\sqrt{d})\) with discriminants \(d<0\).

\begin{theorem}
\label{thm:ImaginaryThreeStage}
Among the finite \(3\)-groups \(G\)
with elementary bicyclic abelianization \(G/G^\prime\simeq C_3\times C_3\) of rank two,
there exist infinitely many non-metabelian groups with GI-action and relation rank \(d_2{G}=2\)
(so-called Schur \(\sigma\)-groups
\cite{KoVe,BBH}),
but only seven of minimal order \(3^8\),
namely \(\langle 6561,i\rangle\) with \(i\in\lbrace 606,616,617,618,620,622,624\rbrace\).
\begin{enumerate}
\item
These are the groups of smallest order
which are admissible as non-metabelian \(3\)-class tower groups \(G\simeq\mathrm{G}_3^\infty{F}\)
of imaginary quadratic fields \(F\) with \(3\)-class group \(\mathrm{Cl}_3{F}\simeq C_3\times C_3\).
\item
Exceptionally, for an imaginary quadratic field \(F\),
the trailing six of these groups are determined by the second order Artin pattern already.
\begin{enumerate}
\item
If \(\mathrm{AP}^{(2)}{F}=(\lbrack 1^2;(32,21,1^3,21)\rbrack,\lbrack 1;(1313)\rbrack)\)
then \(\mathrm{G}_3^\infty{F}\simeq\langle 6561,616\rangle\).
\item
If \(\mathrm{AP}^{(2)}{F}=(\lbrack 1^2;(32,21,1^3,21)\rbrack,\lbrack 1;(2313)\rbrack)\)
then \(\mathrm{G}_3^\infty{F}\simeq\langle 6561,i\rangle\) with \(i\in\lbrace 617,618\rbrace\).
\item
If \(\mathrm{AP}^{(2)}{F}=(\lbrack 1^2;(32,21,21,21)\rbrack,\lbrack 1;(1231)\rbrack)\)
then \(\mathrm{G}_3^\infty{F}\simeq\langle 6561,622\rangle\).
\item
If \(\mathrm{AP}^{(2)}{F}=(\lbrack 1^2;(32,21,21,21)\rbrack,\lbrack 1;(2231)\rbrack)\)
then \(\mathrm{G}_3^\infty{F}\simeq\langle 6561,i\rangle\) with \(i\in\lbrace 620,624\rbrace\).
\end{enumerate}
\item
The actual distribution of these \(3\)-class tower groups \(G\)
among the \(24\,476\) imaginary quadratic fields \(F=\mathbb{Q}(\sqrt{d})\)
with \(3\)-class group \(\mathrm{Cl}_3{F}\simeq C_3\times C_3\) and
discriminants \(-10^7<d<0\) is presented in
Table
\ref{tbl:ImaginaryThreeStage}.
\end{enumerate}
\end{theorem}


\renewcommand{\arraystretch}{1.2}

\begin{table}[ht]
\caption{Frequencies of non-metabelian \(3\)-class tower groups \(G\) for \(-10^7<d<0\)}
\label{tbl:ImaginaryThreeStage}
\begin{center}
\begin{tabular}{|l|r||r|c||l|c||r|}
\hline
 \(G\simeq\)                          & abs. fr. & rel. fr.   & w. r. t.    & type              & \(\varkappa\) & \(\lvert d\rvert_{\text{min}}\) \\
\hline
 \(\langle 6561,616\rangle\)          &  \(760\) & \(3.11\%\) & \(24\,476\) & \(\mathrm{E}.6\)  & \((1313)\)    & \(15\,544\) \\
\hline
 \(\langle 6561,617\rangle\) or       & \(1572\) & \(6.42\%\) & \(24\,476\) & \(\mathrm{E}.14\) & \((2313)\)    & \(16\,627\) \\
 \(\langle 6561,618\rangle\)          &          &            &             &                   &               &             \\
\hline
 \(\langle 6561,622\rangle\)          &  \(798\) & \(3.26\%\) & \(24\,476\) & \(\mathrm{E}.8\)  & \((1231)\)    & \(34\,867\) \\
\hline
 \(\langle 6561,620\rangle\) or       & \(1583\) & \(6.47\%\) & \(24\,476\) & \(\mathrm{E}.9\)  & \((2231)\)    &  \(9\,748\) \\
 \(\langle 6561,624\rangle\)          &          &            &             &                   &               &             \\
\hline
\end{tabular}
\end{center}
\end{table}


\begin{proof}
By a similar but more extensive search than in the proof of Theorem
\ref{thm:ImaginaryTwoStage}.
Data for Table
\ref{tbl:ImaginaryThreeStage}
has been computed by ourselves in June \(2016\)
using MAGMA
\cite{MAGMA}.
\end{proof}


\begin{remark}
\label{rmk:ImaginaryThreeStage}
It should be pointed out that items (1) and (2) of Theorem
\ref{thm:ImaginaryThreeStage}
are \textit{not valid for real quadratic fields},
as documented in
\cite[Thm. 7.8, p. 162, and Thm. 7.12, p. 165]{Ma11b}.

The group \(\langle 6561,606\rangle\)
belongs to the infinite Shafarevich cover of the metabelian group \(\langle 729,45\rangle\)
with respect to imaginary quadratic fields
\cite[Cor. 6.2, p. 301]{Ma7},
\cite{Ma7b}.
It shares a common second order Artin pattern
with all other elements of the Shafarevich cover.
Third order Artin patterns must be used for its identification, as shown in
\cite[Thm. 7.14, p. 168]{Ma11b}.
\end{remark}


\noindent
Now we turn to real quadratic fields \(F=\mathbb{Q}(\sqrt{d})\) with discriminants \(d>0\).

\begin{theorem}
\label{thm:RealThreeStage}
Among the finite \(3\)-groups \(G\)
with elementary bicyclic abelianization \(G/G^\prime\simeq C_3\times C_3\) of rank two,
there exist infinitely many non-metabelian groups with GI-action and relation rank \(d_2{G}=3\)
(so-called Schur\(+1\) \(\sigma\)-groups
\cite{BBH2}),
but only nine of minimal order \(3^7\),
namely \(\langle 2187,i\rangle\) with \(i\in\lbrace 270,271,272,273,284,291,307,308,311\rbrace\).
\begin{enumerate}
\item
These are the groups of smallest order
which are admissible as non-metabelian \(3\)-class tower groups \(G\simeq\mathrm{G}_3^\infty{F}\)
of real quadratic fields \(F\) with \(3\)-class group \(\mathrm{Cl}_3{F}\simeq C_3\times C_3\).
\item
Exceptionally, for a real quadratic field \(F\),
four of these groups are determined by the second order Artin pattern already.
\begin{enumerate}
\item
If \(\mathrm{AP}^{(2)}{F}=(\lbrack 1^2;(2^2,21,1^3,21)\rbrack,\lbrack 1;(0313)\rbrack)\)
then \(\mathrm{G}_3^\infty{F}\simeq\langle 2187,i\rangle\) with \(i\in\lbrace 284,291\rbrace\).
\item
If \(\mathrm{AP}^{(2)}{F}=(\lbrack 1^2;(2^2,21,21,21)\rbrack,\lbrack 1;(0231)\rbrack)\)
then \(\mathrm{G}_3^\infty{F}\simeq\langle 2187,i\rangle\) with \(i\in\lbrace 307,308\rbrace\).
\end{enumerate}
\item
The actual distribution of these \(3\)-class tower groups \(G\)
among the \(415\,698\) real quadratic fields \(F=\mathbb{Q}(\sqrt{d})\)
with \(3\)-class group \(\mathrm{Cl}_3{F}\simeq C_3\times C_3\) and
discriminants \(1<d<10^9\) is presented in
Table
\ref{tbl:RealThreeStage}.
\end{enumerate}
\end{theorem}


\renewcommand{\arraystretch}{1.2}

\begin{table}[ht]
\caption{Frequencies of non-metabelian \(3\)-class tower groups \(G\) for \(0<d<10^9\)}
\label{tbl:RealThreeStage}
\begin{center}
\begin{tabular}{|l|r||r|c||l|c||r|}
\hline
 \(G\simeq\)                          & abs. fr. & rel. fr.   & w. r. t.     & type              & \(\varkappa\) & \(d_{\text{min}}\) \\
\hline
 \(\langle 2187,284\rangle\) or       & \(4318\) & \(1.04\%\) & \(415\,698\) & \(\mathrm{c}.18\) & \((0313)\)    &  \(534\,824\) \\
 \(\langle 2187,291\rangle\)          &          &            &              &                   &               &               \\
\hline
 \(\langle 2187,307\rangle\) or       & \(4377\) & \(1.05\%\) & \(415\,698\) & \(\mathrm{c}.21\) & \((0231)\)    &  \(540\,365\) \\
 \(\langle 2187,308\rangle\)          &          &            &              &                   &               &               \\
\hline
\end{tabular}
\end{center}
\end{table}


\begin{proof}
The claims for transfer kernel type \(\mathrm{c}.18\), \(\varkappa(F)\sim (0313)\),
are a consequence of
\cite[Prp. 7.1, p. 32, Thm. 7.1, p. 33, and Rmk. 7.1, p. 35]{Ma10},
those for type \(\mathrm{c}.21\), \(\varkappa(F)\sim (0231)\), have been proved in
\cite[Prp. 8.1, p. 42, Thm. 8.1, p. 44, and Rmk. 8.2, p. 45]{Ma10}.
A slightly stronger result is the Main Theorem
\cite[Thm. 2.1, p. 22]{Ma10}.
\end{proof}

\begin{remark}
\label{rmk:RealThreeStage}
The groups \(\langle 2187,i\rangle\) with \(i\in\lbrace 270,271,272,273\rbrace\)
are elements of the infinite Shafarevich cover of the metabelian group \(\langle 729,45\rangle\)
with respect to real quadratic fields.

The group \(\langle 2187,311\rangle\)
belongs to the infinite Shafarevich cover of the metabelian group \(\langle 729,57\rangle\)
with respect to real quadratic fields.

These five groups share a common second order Artin pattern
with all other elements of the relevant Shafarevich cover.
Third order Artin patterns must be employed for their identification, as shown in
\cite[Thm. 7.13, p. 167, and Thm. 7.15, p. 169]{Ma11b}.
\end{remark}


\section{Maximal subgroups of \(3\)-groups of coclass one}
\label{s:MaximalSubgroups}
\noindent
Let \((\gamma_i(G))_{i\ge 1}\) be the descending lower central series of the group \(G\),
defined recursively by \(\gamma_1(G):=G\) and
\(\gamma_i(G):=\lbrack\gamma_{i-1}(G),G\rbrack\) for \(i\ge 2\),
in particular, \(\gamma_2(G)=G^\prime\) is the commutator subgroup of \(G\).
A finite \(p\)-group \(G\) is nilpotent with
\(\gamma_1(G)>\gamma_2(G)>\ldots>\gamma_c(G)>\gamma_{c+1}(G)=1\)
for some integer \(c\ge 1\),
which is called the \textit{nilpotency class} \(\mathrm{cl}(G)=c\) of \(G\).
When \(G\) is of order \(p^n\), for some integer \(n\ge 1\),
the \textit{coclass} of \(G\) is defined by \(\mathrm{cc}(G):=n-c\)
and \(\mathrm{lo}(G):=n\) is called the \textit{logarithmic order} of \(G\).

Finite \(3\)-groups \(G\) with coclass \(\mathrm{cc}(G)=1\) were investigated by N. Blackburn
\cite{Bl2}
in \(1958\).
All of these CF-groups,
which exclusively have \textit{cyclic factors} \(\gamma_i(G)/\gamma_{i+1}(G)\)
of their descending central series for \(i\ge 2\),
are necessarily metabelian with second derived subgroup \(G^{\prime\prime}=1\)
and abelian commutator subgroup \(G^\prime\)
and possess abelianization \(G/G^\prime\simeq C_3\times C_3\),
according to Blackburn
\cite{Bl1}.

For the statement of Theorem
\ref{thm:MaxSbgCc1},
we need a precise ordering of the four maximal subgroups \(H_1,\ldots,H_4\) of the group \(G=\langle x,y\rangle\),
which can be generated by two elements \(x,y\),
according to the Burnside basis theorem.
For this purpose, we select the generators \(x,y\) such that
\begin{equation}
\label{eqn:MaximalSubgroups}
H_1=\langle y,G^\prime\rangle,\quad
H_2=\langle x,G^\prime\rangle,\quad
H_3=\langle xy,G^\prime\rangle,\quad
H_4=\langle xy^2,G^\prime\rangle,
\end{equation}
and \(H_1=\chi_2(G)\),
provided that \(G\) is of nilpotency class \(\mathrm{cl}(G)\ge 3\).
Here we denote by
\begin{equation}
\label{eqn:TwoStepCentralizer}
\chi_2(G):=\lbrace g\in G\mid (\forall\ h\in\gamma_2(G))\ \lbrack g,h\rbrack\in\gamma_4(G)\rbrace
\end{equation}
the \textit{two-step centralizer} of \(G^\prime\) in \(G\).


\subsection{Parametrized presentations of metabelian \(3\)-groups}
\label{ss:Presentations}
\noindent
The identification of the groups will be achieved with the aid of
parametrized polycyclic power-commutator presentations, as given by
Blackburn
\cite{Bl2},
Miech
\cite{Mi},
and Nebelung
\cite{Ne}:
\begin{equation}
\label{eqn:Presentation}
\begin{aligned}
G_a^n(z,w) := \langle x,y,s_2,\ldots,s_{n-1}\mid s_2=\lbrack y,x\rbrack,\ (\forall_{i=3}^n)\ s_i=\lbrack s_{i-1},x\rbrack,\ s_n=1,\ \lbrack y,s_2\rbrack=s_{n-1}^a, \\
(\forall_{i=3}^{n-1})\ \lbrack y,s_i\rbrack=1,\ x^3=s_{n-1}^w,\ y^3s_2^3s_3=s_{n-1}^z,\ (\forall_{i=2}^{n-3})\ s_i^3s_{i+1}^3s_{i+2}=1,\ s_{n-2}^3=s_{n-1}^3=1\ \rangle,
\end{aligned}
\end{equation}
where \(a\in\lbrace 0,1\rbrace\) and \(w,z\in\lbrace -1,0,1\rbrace\) are bounded parameters,
and the \textit{index of nilpotency} \(n=\mathrm{cl}(G)+1=\mathrm{cl}(G)+\mathrm{cc}(G)=\log_3(\mathrm{ord}(G))=:\mathrm{lo}(G)\) is an unbounded parameter.


The following lemma generalizes relations for second and third powers of generators in
\cite[Lem. 3.1]{Ma17},
\cite{Ma17b}.

\begin{lemma}
\label{lem:PowerRelations}
Let \(G=\langle x,y\rangle\) be a finite \(3\)-group with two generators \(x,y\in G\).
Denote by \(s_2:=\lbrack y,x\rbrack\) the main commutator, and by
\(s_3:=\lbrack s_2,x\rbrack\) and \(t_3:=\lbrack s_2,y\rbrack\) the two iterated commutators.
Then the second and third power of the element \(xy\), respectively \(xy^2\), are given by
\begin{equation}
\label{eqn:PowerRelations}
\begin{aligned}
(xy)^2   &= x^2y^2s_2t_3     & \text{ and } (xy)^3   &= x^3y^3s_2^3s_3t_3^2, \text{ respectively} \\
(xy^2)^2 &= x^2y^4s_2^2t_3^2 & \text{ and } (xy^2)^3 &= x^3y^6s_2^6s_3^2t_3^2,
\end{aligned}
\end{equation}
provided that \(t_3\in\zeta(G)\) is central, \(t_3^3=1\), and \(\lbrack s_3,y\rbrack=1\).
\end{lemma}

\begin{proof}
We begin by preparing three commutator relations:
\begin{equation}
\label{eqn:CommutatorRelations}
yx=xy\lbrack y,x\rbrack=xys_2,\quad
s_2x=xs_2\lbrack s_2,x\rbrack=xs_2s_3,\quad
\text{ and } \quad
s_2y=ys_2\lbrack s_2,y\rbrack=ys_2t_3.
\end{equation}
Now we prove the power relations
by expanding the power expressions by iterated substitution of the commutator relations in Formula
\eqref{eqn:CommutatorRelations},
always observing that \(t_3\) belongs to the centre, \(t_3^3=1\), and \(s_3y=ys_3\) commute:

\begin{equation*}
(xy)^2=xyxy=xxys_2y=x^2yys_2t_3=x^2y^2s_2t_3, \text{ and thus}
\end{equation*}

\begin{equation*}
\begin{aligned}
(xy)^3 &= (xy)^2xy=x^2y^2s_2t_3xy=x^2y^2s_2xyt_3=x^2yyxs_2s_3yt_3=x^2yxys_2s_2ys_3t_3= \\
&= x^2xys_2ys_2ys_2t_3s_3t_3=x^3yys_2t_3ys_2t_3s_2s_3t_3^2=x^3y^2s_2ys_2s_2s_3t_3^4= \\
&= x^3y^2ys_2t_3s_2^2s_3t_3=x^3y^3s_2^3s_3t_3^2, \text{ respectively}
\end{aligned}
\end{equation*}

\begin{equation*}
\begin{aligned}
(xy^2)^2 &= xyyxyy=xyxys_2yy=xxys_2yys_2t_3y=x^2yys_2t_3ys_2yt_3=x^2y^2s_2yys_2t_3t_3^2= \\
&= x^2y^2ys_2t_3ys_2t_3^3=x^2y^3s_2ys_2t_3=x^2y^3ys_2t_3s_2t_3=x^2y^4s_2^2t_3^2, \text{ and thus}
\end{aligned}
\end{equation*}

\begin{equation*}
\begin{aligned}
(xy^2)^3 &= (xy^2)^2xy^2=x^2y^4s_2^2t_3^2xy^2=x^2y^4s_2s_2xyyt_3^2=x^2y^4s_2xs_2s_3yyt_3^2= \\
&= x^2yyy yx s_2s_3 s_2y ys_3t_3^2=x^2yy yx ys_2 s_2y s_2t_3ys_3^2t_3^2=x^2yyxys_2ys_2ys_2t_3s_2ys_3^2t_3^3= \\
&= x^2yxys_2yys_2t_3ys_2t_3s_2ys_2t_3s_3^2t_3^4=x^2xys_2yys_2t_3ys_2ys_2t_3^2ys_2t_3s_2s_3^2t_3^2= \\
&= x^3yys_2t_3ys_2yys_2t_3s_2t_3^3ys_2^2s_3^2t_3^3=x^3y^2s_2ys_2yys_2t_3^2s_2ys_2^2s_3^2= \\
&= x^3y^2ys_2t_3ys_2t_3ys_2t_3^2ys_2t_3s_2^2s_3^2=x^3y^3s_2ys_2t_3^2ys_2yt_3^3s_2^3s_3^2=x^3y^3ys_2t_3s_2t_3^2yys_2t_3s_2^3s_3^2= \\
&= x^3y^4s_2s_2yys_2t_3^4s_2^3s_3^2=x^3y^4s_2s_2yys_2^4s_3^2t_3=x^3y^4s_2ys_2t_3ys_2^4s_3^2t_3=x^3y^4s_2ys_2ys_2^4s_3^2t_3^2= \\
&= x^3y^4ys_2t_3ys_2t_3s_2^4s_3^2t_3^2=x^3y^5s_2ys_2t_3^2s_2^4s_3^2t_3^2=x^3y^5y s_2t_3s_2^5s_3^2t_3^4=x^3y^6s_2^6s_3^2t_3^2.
\end{aligned}
\end{equation*}
\end{proof}


\begin{theorem}
\label{thm:MaxSbgCc1}
\noindent
Let \(G=\langle x,y\rangle\simeq G_a^n(z,w)\) be a finite \(3\)-group of coclass \(\mathrm{cc}(G)=1\)
and order \(\lvert G\rvert=3^n\) with generators \(x,y\)
such that \(y\in\chi_2(G)\) is contained in the two-step centralizer of \(G\),
whereas \(x\in G\setminus\chi_2(G)\),
given by a polycyclic power commutator presentation with parameters
\(a\in\lbrace 0,1\rbrace\), \(w,z\in\lbrace -1,0,1\rbrace\), and index of nilpotency \(n\ge 4\).

Then three of the four maximal subgroups, \(H_i=\langle xy^{i-2},G^\prime\rangle<G\), \(2\le i\le 4\),
are non-abelian \(3\)-groups of coclass \(\mathrm{cc}(H_i)=1\), as listed in Table
\ref{tbl:MaxSbgCc1}
in dependence on the parameters \(n,a,z,w\).

The supplementary Table
\ref{tbl:MaxSbgExtraSpecial}
shows the abelian maximal subgroups of the
remaining two extra special \(3\)-group of coclass \(\mathrm{cc}(G)=1\)
and order \(\lvert G\rvert=3^3\).
\end{theorem}


\renewcommand{\arraystretch}{1.2}

\begin{table}[ht]
\caption{Non-abelian maximal subgroups \(H_i<G\) of \(3\)-groups \(G\) of coclass \(1\)}
\label{tbl:MaxSbgCc1}
\begin{center}
\begin{tabular}{|c|c|c|c|c||c|c|c|}
\hline
 \(G\simeq\)     & \(n\)     & \(a\) & \(z\)  & \(w\)  & \(H_2=\langle x,G^\prime\rangle\) & \(H_3=\langle xy,G^\prime\rangle\) & \(H_4=\langle xy^2,G^\prime\rangle\) \\
\hline
 \(G_0^n(0,0)\)  & \(\ge 4\) & \(0\) & \(0\)  & \(0\)  & \(\simeq G_0^{n-1}(0,0)\)         & \(\simeq G_0^{n-1}(0,0)\)          & \(\simeq G_0^{n-1}(0,0)\)            \\
 \(G_0^n(0,1)\)  & \(\ge 4\) & \(0\) & \(0\)  & \(1\)  & \(\simeq G_0^{n-1}(0,1)\)         & \(\simeq G_0^{n-1}(0,1)\)          & \(\simeq G_0^{n-1}(0,1)\)            \\
 \(G_0^n(1,0)\)  & \(\ge 4\) & \(0\) & \(1\)  & \(0\)  & \(\simeq G_0^{n-1}(0,0)\)         & \(\simeq G_0^{n-1}(0,1)\)          & \(\simeq G_0^{n-1}(0,1)\)            \\
 \(G_0^n(-1,0)\) & \(\ge 4\) & \(0\) & \(-1\) & \(0\)  & \(\simeq G_0^{n-1}(0,0)\)         & \(\simeq G_0^{n-1}(0,1)\)          & \(\simeq G_0^{n-1}(0,1)\)            \\
\hline
 \(G_1^n(0,-1)\) & \(\ge 5\) & \(1\) & \(0\)  & \(-1\) & \(\simeq G_0^{n-1}(0,1)\)         & \(\simeq G_0^{n-1}(0,0)\)          & \(\simeq G_0^{n-1}(0,0)\)            \\
 \(G_1^n(0,0)\)  & \(\ge 5\) & \(1\) & \(0\)  & \(0\)  & \(\simeq G_0^{n-1}(0,0)\)         & \(\simeq G_0^{n-1}(0,1)\)          & \(\simeq G_0^{n-1}(0,1)\)            \\
 \(G_1^n(0,1)\)  & \(\ge 5\) & \(1\) & \(0\)  & \(1\)  & \(\simeq G_0^{n-1}(0,1)\)         & \(\simeq G_0^{n-1}(0,1)\)          & \(\simeq G_0^{n-1}(0,1)\)            \\
\hline
\end{tabular}
\end{center}
\end{table}


\renewcommand{\arraystretch}{1.2}

\begin{table}[ht]
\caption{Abelian maximal subgroups \(H_i<G\) of extra special \(3\)-groups \(G\)}
\label{tbl:MaxSbgExtraSpecial}
\begin{center}
\begin{tabular}{|c|c|c|c|c||c|c|c|c|}
\hline
 \(G\simeq\)     & \(n\) & \(a\) & \(z\)  & \(w\)  & \(H_1=\langle y,G^\prime\rangle\) & \(H_2=\langle x,G^\prime\rangle\) & \(H_3=\langle xy,G^\prime\rangle\) & \(H_4=\langle xy^2,G^\prime\rangle\) \\
\hline
 \(G_0^3(0,0)\)  & \(3\) & \(0\) & \(0\)  & \(0\)  & \(\simeq C_3\times C_3\)          & \(\simeq C_3\times C_3\)          & \(\simeq C_3\times C_3\)            & \(\simeq C_3\times C_3\)            \\
 \(G_0^3(0,1)\)  & \(3\) & \(0\) & \(0\)  & \(1\)  & \(\simeq C_3\times C_3\)          & \(\simeq C_9\)                    & \(\simeq C_9\)                      & \(\simeq C_9\)                      \\
\hline
\end{tabular}
\end{center}
\end{table}



\begin{proof}
For an index of nilpotency \(n\ge 4\),
the first maximal subgroup \(H_1=\langle y,G^\prime\rangle\) of \(G\)
coincides with the two-step centralizer \(\chi_2(G)\) of \(G\),
which is a \textit{nearly homocyclic abelian} \(3\)-group \(A(3,n-1)\) of order \(3^{n-1}\), when \(a=0\).
For \(a=1\), we have \(H_1/H_1^\prime\simeq A(3,n-1)\).

We transform all relations of the group \(G\simeq G_a^n(z,w)\)
into relations of the remaining three maximal subgroups \(H\simeq G_\alpha^{n-1}(\zeta,\omega)\) of \(G\).

The \textit{polycyclic commutator relations}
\(s_2=\lbrack y,x\rbrack\), \(s_i=\lbrack s_{i-1},x\rbrack\) for \(3\le i\le n\),
and the \textit{nilpotency relation} \(s_n=1\) for the group \(G=\langle x,y\rangle\),
with lower central series \(\gamma_i{G}=\langle s_i,\gamma_{i+1}{G}\rangle\) for \(i\ge 2\),
can be used immediately for the subgroup \(H_2=\langle x,G^\prime\rangle=\langle x,s_2\rangle\)
with lower central series \(\gamma_i{H_2}=\langle t_i,\gamma_{i+1}{H_2}\rangle\),
where \(t_i:=s_{i+1}\) for \(i\ge 2\), and \(t_{n-1}=1\).

For the lower central series of \(H_3=\langle xy,G^\prime\rangle\) and \(H_4=\langle xy^2,G^\prime\rangle\),
we must employ the \textit{main commutator relation} \(\lbrack y,s_2\rbrack=s_{n-1}^a\),
and \(\lbrack y,s_i\rbrack=1\) for \(i\ge 3\).
According to the \textit{right product rule} for commutators, we have
\(\lbrack s_{i-1},xy\rbrack=\lbrack s_{i-1},y\rbrack\cdot\lbrack s_{i-1},x\rbrack^y=1\cdot s_i^y=s_i\lbrack s_i,y\rbrack=s_i\cdot 1=s_i\),
for \(i\ge 4\), but
\(\lbrack s_2,xy\rbrack=\lbrack s_2,y\rbrack\cdot\lbrack s_2,x\rbrack^y=s_{n-1}^{-a}s_3^y=s_{n-1}^{-a}s_3\lbrack s_3,y\rbrack=s_{n-1}^{-a}s_3\),
and in a similar fashion
\(\lbrack s_{i-1},xy^2\rbrack=\lbrack s_{i-1},y\rbrack\cdot\lbrack s_{i-1},xy\rbrack^y=1\cdot s_i^y=s_i\lbrack s_i,y\rbrack=s_i\cdot 1=s_i\),
for \(i\ge 4\), but again exceptionally
\(\lbrack s_2,xy^2\rbrack=\lbrack s_2,y\rbrack\cdot\lbrack s_2,xy\rbrack^y=s_{n-1}^{-a}y^{-1}s_{n-1}^{-a}s_3y=s_{n-1}^{-2a}s_3=s_{n-1}^as_3\).
For \(a=1\), the \textit{left product rule} for commutators shows
\(\lbrack s_{n-1}^{\mp 1}s_3,xy^{\pm 1}\rbrack=\lbrack s_{n-1}^{\mp 1},xy^{\pm 1}\rbrack^{s_3}\cdot\lbrack s_3,xy^{\pm 1}\rbrack=s_4\),
that is, the slight anomaly for the main commutator disappears in the next step.
Thus, the lower central series is \(\gamma_i{H_j}=\langle t_i,\gamma_{i+1}{H_j}\rangle\) for \(i\ge 2\), \(3\le j\le 4\),
where generally \(t_i:=s_{i+1}\) for \(i\ge 3\), and \(t_2:=s_3\) for \(a=0\), \(t_2:=s_{n-1}^{2-j}s_3\) for \(a=1\).
In particular, \(H_3=\langle xy,s_2\rangle\) and \(H_4=\langle xy^2,s_2\rangle\).

The main commutator relation for all three subgroups \(H_2,H_3,H_4\) of any group \(G\simeq G_a^n(z,w)\) with \(n\ge 4\) is
\(\lbrack s_2,t_2\rbrack=1=t_{n-2}^\alpha\), that is \(\alpha=0\), generally,
and it remains to determine \(\zeta,\omega\).

For this purpose, we come to the \textit{power relations} of \(G\),
\(x^3=s_{n-1}^w\), \(y^3s_2^3s_3=s_{n-1}^z\), and \(s_i^3s_{i+1}^3s_{i+2}=1\) for \(i\ge 2\),
supplemented by 
\((xy)^3=x^3y^3s_2^3s_3s_{n-1}^{-2a}=s_{n-1}^ws_{n-1}^zs_{n-1}^{-2a}\) and \((xy^2)^3=x^3(y^3s_2^3s_3)^2s_{n-1}^{-2a}=s_{n-1}^ws_{n-1}^{2z}s_{n-1}^{-2a}\),
and we use these relations to determine \(\zeta,\omega\) in dependence on \(w,z,a\).
Generally, we have \(s_2^3t_2^3t_3=s_2^3s_3^3s_4=1\) for \(a=0\),
\(s_2^3t_2^3t_3=s_2^3s_{n-1}^{3(2-j)}s_3^3s_4=s_2^3s_3^3s_4=1\) for \(a=1\),
and thus uniformly \(\zeta=0\).

For \(G_0^n(0,0)\), we uniformly have \(x^3=(xy)^3=(xy^2)^3=1\), and thus \(\omega=0\) for all three subgroups.
For \(G_0^n(0,1)\), we uniformly have \(x^3=(xy)^3=(xy^2)^3=s_{n-1}\), and thus \(\omega=1\) for all three subgroups.
For \(G_0^n(\pm 1,0)\), we have \(x^3=1\), but \((xy)^3=s_{n-1}^{\pm 1}\), \((xy^2)^3=s_{n-1}^{\pm 2}=s_{n-1}^{\mp 1}\),
and thus \(\omega=0\) for \(H_2\) but \(\omega=1\) for \(H_3,H_4\), since \(G_0^n(0,-1)\simeq G_0^n(0,1)\).

For \(G_1^n(0,-1)\), we have \(x^3=s_{n-1}^{-1}\), but \((xy)^3=(xy^2)^3=s_{n-1}^{-3}=1\),
and thus \(\omega=1\) for \(H_2\) but \(\omega=0\) for \(H_3,H_4\).
For \(G_1^n(0,0)\), we have \(x^3=1\), but \((xy)^3=(xy^2)^3=s_{n-1}^{-2}=s_{n-1}\),
and thus \(\omega=0\) for \(H_2\) but \(\omega=1\) for \(H_3,H_4\).
For \(G_1^n(0,1)\), we have \(x^3=s_{n-1}\), \((xy)^3=(xy^2)^3=s_{n-1}^{-1}\),
and thus \(\omega=1\) for all three subgroups, again observing that \(G_0^n(0,-1)\simeq G_0^n(0,1)\).

The only \(3\)-groups \(G\) of coclass \(\mathrm{cc}(G)=1\) and order \(\lvert G\rvert=3^3\)
are the two extra special groups \(G_0^3(0,0)\) and \(G_0^3(0,1)\).
Since \(t_2=s_3=1\), all their four maximal subgroups,
\(H_1=\langle y,s_2\rangle\), \(H_2=\langle x,s_2\rangle\), \(H_3=\langle xy,s_2\rangle\), \(H_4=\langle xy^2,s_2\rangle\), 
are abelian.
For \(w=z=0\), \(s_2\) is independent of the other generator, and \(H_i\simeq C_3\times C_3\) for \(1\le i\le 4\).
However, for \(w=1\), \(z=0\), we have \(x^3=(xy)^3=(xy^2)^3=s_2\), \(s_2^3=1\), and thus \(H_2\simeq H_3\simeq H_4\simeq C_9\),
whereas \(H_1\simeq C_3\times C_3\).
\end{proof}



\section{A general theorem for arbitrary base fields}
\label{s:General}
\noindent
Suppose that \(p\) is a prime,
\(F\) is an algebraic number field
with non-trivial \(p\)-class group \(\mathrm{Cl}_p{F}>1\),
and \(E\) is one of the unramified abelian \(p\)-extensions of \(F\).
We show that, even in this general situation,
a finite \(p\)-class tower of \(F\)
exerts a very severe restriction on the \(p\)-class tower of \(E\).

\begin{theorem}
\label{thm:General}
Assume that \(F\) possesses a \(p\)-class tower \(F_p^{(\infty)}=F_p^{(n)}\)
of exact length \(\ell_p{F}=n\) for some integer \(n\ge 1\).
Then the Galois group \(\mathrm{Gal}(E_p^{(\infty)}/E)\) of the \(p\)-class tower of \(E\)
is a subgroup of index \(\lbrack E:F\rbrack\)
of the \(p\)-class tower group \(\mathrm{Gal}(F_p^{(\infty)}/F)\) of \(F\)
and the length of the \(p\)-class tower of \(E\) is bounded by \(\ell_p{E}\le n\).
\end{theorem}

\begin{proof}
According to the assumptions,
there exists a tower of field extensions,
\[F<E\le F_p^{(1)}\le E_p^{(1)}\le F_p^{(2)}\le E_p^{(2)}\le\ldots\le F_p^{(n)}\le E_p^{(n)}\le F_p^{(n+1)},\]
where \(\ell_p{F}=n\) enforces the coincidence \(F_p^{(n)}=E_p^{(n)}=F_p^{(n+1)}\) of the trailing three fields.
Since \(\mathrm{Gal}(F_p^{(n)}/F)/\mathrm{Gal}(F_p^{(n)}/E)\simeq\mathrm{Gal}(E/F)\),
the group index of \(\mathrm{Gal}(E_p^{(n)}/E)=\mathrm{Gal}(F_p^{(n)}/E)\) in \(\mathrm{Gal}(F_p^{(n)}/F)\)
is equal to the field degree \(\lbrack E:F\rbrack\)
and \(\mathrm{Gal}(E_p^{(\infty)}/E)=\mathrm{Gal}(E_p^{(n)}/E)\)
is a subgroup of index \(\lbrack E:F\rbrack\) of \(\mathrm{Gal}(F_p^{(n)}/F)=\mathrm{Gal}(F_p^{(\infty)}/F)\).
The equality \(E_p^{(n)}=E_p^{(n+1)}\) implies the bound \(\ell_p{E}\le n\).
\end{proof}

\noindent
We shall apply Theorem
\ref{thm:General}
to the situation where \(p=3\), \(n=2\),
and \(E\) is an unramified cyclic cubic extension of \(F\),
whence \(\mathrm{Gal}(E_3^{(\infty)}/E)\) is a maximal subgroup of \(\mathrm{Gal}(F_3^{(\infty)}/F)\).



\subsection{Application to quadratic base fields}
\label{ss:Quadratic}


\begin{proposition}
\label{prp:KernelTypeD10}
\noindent
Let \(G\) be a finite \(3\)-group
with elementary bicyclic abelianization \(G/G^\prime\simeq C_3\times C_3\).
Then the following conditions are equivalent:
\begin{enumerate}
\item
The transfer kernel type of \(G\) is \(\mathrm{D}.10\), \(\varkappa(G)\sim (2241)\).
\item
The abelian quotient invariants of the four maximal subgroups \(H_1,\ldots,H_4\) of \(G\)
are \(\tau(G)\sim(21,21,1^3,21)\).
\item
The isomorphism types of the four maximal subgroups of \(G\) are
\(H_1\simeq H_2\simeq H_4\simeq\langle 3^4,3\rangle\) and \(H_3\simeq\langle 3^4,13\rangle\).
\item
The group \(G\) is isomorphic to the Schur \(\sigma\)-group \(\langle 3^5,5\rangle\) with relation rank \(d_2=2\).
\end{enumerate}
\end{proposition}

\begin{proof}
We put \(G:=\langle 243,5\rangle\) and use the presentation
\cite{MAGMA}
\[G=\langle x,y,s_2,s_3,t_3\mid s_2=\lbrack y,x\rbrack, s_3=\lbrack s_2,x\rbrack, t_3=\lbrack s_2,y\rbrack, x^3=s_3, y^3=s_3\rangle.\]
Then we obtain the maximal subgroups \\
\(H_1=\langle y,G^\prime\rangle=\langle y,s_2,s_3\rangle\), since \(t_3=\lbrack s_2,y\rbrack\), \\
\(H_2=\langle x,G^\prime\rangle=\langle x,s_2,t_3\rangle\), since \(s_3=\lbrack s_2,x\rbrack\), \\
\(H_3=\langle xy,G^\prime\rangle=\langle xy,s_2,s_3\rangle\), since \(\lbrack s_2,xy\rbrack=s_3t_3\), \\
\(H_4=\langle xy^2,G^\prime\rangle=\langle xy^2,s_2,s_3\rangle\), since \(\lbrack s_2,xy^2\rbrack=s_3t_3^2\). \\
Using Lemma
\ref{lem:PowerRelations},
and comparing to the abstract presentations
\cite{MAGMA} \\
\(\langle 81,3\rangle=\langle \xi,\upsilon,\sigma_2,\tau\mid\sigma_2=\lbrack \upsilon,\xi\rbrack, \tau=\xi^3\rangle\) and \\
\(\langle 81,13\rangle=\langle \xi,\upsilon,\zeta,\sigma_2\mid\sigma_2=\lbrack \upsilon,\xi\rbrack,\xi^3=\sigma_2,\upsilon^3=\zeta^3=1\rangle\), \\
we conclude \\
\(H_1=\langle y,s_2,s_3\rangle=\langle y,s_2\rangle\simeq\langle 81,3\rangle\), since \(y^3=s_3\ne\lbrack s_2,y\rbrack=t_3\), \\
\(H_2=\langle x,s_2,t_3\rangle\simeq\langle 81,13\rangle\), since \(x^3=s_3=\lbrack s_2,x\rbrack\), \\
\(H_3=\langle xy,s_2,s_3\rangle=\langle xy,s_2\rangle\simeq\langle 81,3\rangle\), since \((xy)^3=t_3^2\ne\lbrack s_2,xy\rbrack=s_3t_3\), \\
\(H_4=\langle xy^2,s_2,s_3\rangle=\langle xy^2,s_2\rangle\simeq\langle 81,3\rangle\), since \((xy^2)^3=s_3^2t_3^2\ne\lbrack s_2,xy^2\rbrack=s_3t_3^2\).
\end{proof}


\begin{theorem}
\label{thm:KernelTypeD10}
\noindent
Let \(F=\mathbb{Q}(\sqrt{d})\) be a quadratic field
with elementary bicyclic \(3\)-class group \(\mathrm{Cl}_3{F}\simeq C_3\times C_3\).
Then the following conditions are equivalent:
\begin{enumerate}
\item
The transfer kernel type of \(F\) is \(\mathrm{D}.10\), \(\varkappa(F)\sim (2241)\).
\item
The abelian type invariants of the \(3\)-class groups \(\mathrm{Cl}_3{E_i}\) of the four unramified cyclic cubic extensions \(E_i/F\)
are \(\tau(F)\sim(21,21,1^3,21)\).
\item
The second \(3\)-class group \(\mathrm{G}_3^2{F}\) of \(F\) has the maximal subgroups
\(H_1\simeq H_2\simeq H_4\simeq\langle 3^4,3\rangle\) and \(H_3\simeq\langle 3^4,13\rangle\).
\item
The \(3\)-class tower group \(\mathrm{G}_3^\infty{F}\) of \(F\) is the Schur \(\sigma\)-group \(\langle 3^5,5\rangle\) with relation rank \(d_2=2\).
\end{enumerate}
\end{theorem}

\begin{proof}
The claims follow from Proposition
\ref{prp:KernelTypeD10}
by applying the Successive Approximation Theorem
\ref{thm:SuccessiveApproximation}
of first order.
\end{proof}


\begin{corollary}
\label{cor:KernelTypeD10}
\noindent
Let \(F\) be a quadratic field which satisfies one of the equivalent conditions in Theorem
\ref{thm:KernelTypeD10}.
Then the length of the \(3\)-class tower of \(F\) is \(\ell_3{F}=2\).
The four unramified cyclic cubic extensions \(E_i/F\)
are absolutely dihedral of degree \(6\),
with torsionfree Dirichlet unit rank \(r\ge 2\),
and possess \(3\)-class towers of length \(\ell_3{E_i}=2\).
More precisely,
\(\mathrm{Cl}_3{E_3}\simeq C_3\times C_3\times C_3\)
and \(\mathrm{G}_3^\infty{E_3}\simeq\langle 3^4,13\rangle\) with relation rank \(d_2=5\),
but \(\mathrm{Cl}_3{E_i}\simeq C_9\times C_3\)
and \(\mathrm{G}_3^\infty{E_i}\simeq\langle 3^4,3\rangle\) with relation rank  \(d_2=4\)
for \(i\in\lbrace 1,2,4\rbrace\).
\end{corollary}

\begin{proof}
This is a consequence of Theorems
\ref{thm:General}
and
\ref{thm:KernelTypeD10},
satisfying the Shafarevich theorem.
\end{proof}


\begin{proposition}
\label{prp:KernelTypeD5}
\noindent
Let \(G\) be a finite \(3\)-group
with elementary bicyclic abelianization \(G/G^\prime\simeq C_3\times C_3\).
Then the following conditions are equivalent:
\begin{enumerate}
\item
The transfer kernel type of \(G\) is \(\mathrm{D}.5\), \(\varkappa(G)\sim (4224)\).
\item
The abelian quotient invariants of the four maximal subgroups \(H_1,\ldots,H_4\) of \(G\)
are \(\tau(G)\sim(1^3,21,1^3,21)\).
\item
The isomorphism types of the four maximal subgroups of \(G\) are
\(H_1\simeq H_3\simeq\langle 3^4,13\rangle\) and \(H_2\simeq H_4\simeq\langle 3^4,3\rangle\).
\item
The group \(G\) is isomorphic to the Schur \(\sigma\)-group \(\langle 3^5,7\rangle\) with relation rank \(d_2=2\).
\end{enumerate}
\end{proposition}

\begin{proof}
We put \(G:=\langle 243,7\rangle\) and use the presentation
\cite{MAGMA}
\[G=\langle x,y,s_2,s_3,t_3\mid s_2=\lbrack y,x\rbrack, s_3=\lbrack s_2,x\rbrack, t_3=\lbrack s_2,y\rbrack, x^3=s_3, y^3=s_3^2\rangle.\]
Similarly as in Proposition
\ref{prp:KernelTypeD10},
we obtain the maximal subgroups \\
\(H_1=\langle y,G^\prime\rangle=\langle y,s_2,s_3\rangle\),
\(H_2=\langle x,G^\prime\rangle=\langle x,s_2,t_3\rangle\), \\
\(H_3=\langle xy,G^\prime\rangle=\langle xy,s_2,s_3\rangle\), and
\(H_4=\langle xy^2,G^\prime\rangle=\langle xy^2,s_2,s_3\rangle\). \\
Using Lemma
\ref{lem:PowerRelations},
and comparing to the abstract presentations \\
\(\langle 81,3\rangle=\langle \xi,\upsilon,\sigma_2,\tau\mid\sigma_2=\lbrack \upsilon,\xi\rbrack, \tau=\xi^3\rangle\) and \\
\(\langle 81,13\rangle=\langle \xi,\upsilon,\zeta,\sigma_2\mid\sigma_2=\lbrack \upsilon,\xi\rbrack,\xi^3=\sigma_2,\upsilon^3=\zeta^3=1\rangle\), \\
we conclude \\
\(H_1=\langle y,s_2,s_3\rangle=\langle y,s_2\rangle\simeq\langle 81,3\rangle\), since \(y^3=s_3^2\ne\lbrack s_2,y\rbrack=t_3\), \\
\(H_2=\langle x,s_2,t_3\rangle\simeq\langle 81,13\rangle\), since \(x^3=s_3=\lbrack s_2,x\rbrack\), \\
\(H_3=\langle xy,s_2,s_3\rangle=\langle xy,s_2\rangle\simeq\langle 81,3\rangle\), since \((xy)^3=s_3t_3^2\ne\lbrack s_2,xy\rbrack=s_3t_3\), \\
\(H_4=\langle xy^2,s_2,s_3\rangle\simeq\langle 81,13\rangle\), since \((xy^2)^3=s_3t_3^2=\lbrack s_2,xy^2\rbrack\).
\end{proof}


\begin{theorem}
\label{thm:KernelTypeD5}
\noindent
Let \(F=\mathbb{Q}(\sqrt{d})\) be a quadratic field
with elementary bicyclic \(3\)-class group \(\mathrm{Cl}_3{F}\simeq C_3\times C_3\).
Then the following conditions are equivalent:
\begin{enumerate}
\item
The transfer kernel type of \(F\) is \(\mathrm{D}.5\), \(\varkappa(F)\sim (4224)\).
\item
The abelian type invariants of the \(3\)-class groups \(\mathrm{Cl}_3{E_i}\) of the four unramified cyclic cubic extensions \(E_i/F\)
are \(\tau(F)\sim(1^3,21,1^3,21)\).
\item
The second \(3\)-class group \(\mathrm{G}_3^2{F}\) of \(F\) has the maximal subgroups
\(H_1\simeq H_3\simeq\langle 3^4,13\rangle\) and \(H_2\simeq H_4\simeq\langle 3^4,3\rangle\).
\item
The \(3\)-class tower group \(\mathrm{G}_3^\infty{F}\) of \(F\) is the Schur \(\sigma\)-group \(\langle 3^5,7\rangle\) with relation rank \(d_2=2\).
\end{enumerate}
\end{theorem}

\begin{proof}
The claims follow from Proposition
\ref{prp:KernelTypeD5}
by applying the Successive Approximation Theorem
\ref{thm:SuccessiveApproximation}
of first order.
\end{proof}


\begin{corollary}
\label{cor:KernelTypeD5}
\noindent
Let \(F\) be a quadratic field which satisfies one of the equivalent conditions in Theorem
\ref{thm:KernelTypeD5}.
Then the length of the \(3\)-class tower of \(F\) is \(\ell_3{F}=2\).
The four unramified cyclic cubic extensions \(E_i/F\)
are absolutely dihedral of degree \(6\),
with torsionfree Dirichlet unit rank \(r\ge 2\),
and possess \(3\)-class towers of length \(\ell_3{E_i}=2\).
More precisely,
\(\mathrm{Cl}_3{E_i}\simeq C_3\times C_3\times C_3\)
and \(\mathrm{G}_3^\infty{E_i}\simeq\langle 3^4,13\rangle\) with relation rank \(d_2=5\)
for \(i\in\lbrace 1,3\rbrace\),
but \(\mathrm{Cl}_3{E_i}\simeq C_9\times C_3\)
and \(\mathrm{G}_3^\infty{E_i}\simeq\langle 3^4,3\rangle\) with relation rank  \(d_2=4\)
for \(i\in\lbrace 2,4\rbrace\).
\end{corollary}

\begin{proof}
This is a consequence of Theorems
\ref{thm:General}
and
\ref{thm:KernelTypeD5},
satisfying the Shafarevich theorem.
\end{proof}

\subsection{Application to dihedral fields}
\label{ss:Dihedral}
\noindent
We recall that a dihedral field \(E\) of degree \(6\)
is an absolute Galois extension \(E/\mathbb{Q}\)
with group \(\mathrm{Gal}(E/\mathbb{Q})=\langle\sigma,\tau\mid\sigma^3=\tau^2=1,\sigma\tau=\tau\sigma^{-1}\rangle\).
It is a cyclic cubic relative extension \(E/F\) of its unique quadratic subfield \(F=E^\sigma\),
and it contains three isomorphic, conjugate non-Galois cubic subfields \(L=E^\tau\), \(L^\sigma\), \(L^{\sigma^2}\).
The conductor \(c\) of \(E/F\) is a nearly squarefree positive integer with special prime factors,
and the discriminants satisfy the relations \(d_E=c^4d_F^3\) and \(d_L=c^2d_F\).
Here, we shall always be concerned with unramified extensions, characterized by the conductor \(c=1\),
and thus \(d_E=d_F^3\), a perfect cube, and equal \(d_L=d_F\).


\subsubsection{Totally complex dihedral fields}
\label{sss:TotallyComplexDihedral}
\noindent
The computational information on \(3\)-tower groups \(G:=\mathrm{G}_3^\infty{F}\)
of imaginary quadratic fields \(F\) in Table
\ref{tbl:ImaginaryTwoStage}
admits the purely theoretical deduction
of impressive statistics for \(3\)-tower groups \(S:=\mathrm{G}_3^\infty{E}\)
of totally complex dihedral fields \(E\) in Table
\ref{tbl:TotallyComplexDihedral}
by means of the Corollaries
\ref{cor:KernelTypeD10}
and
\ref{cor:KernelTypeD5}.
We use the crucial new insight that the groups \(S\triangleleft G\) are maximal subgroups of \(G\),
because the extensions \(E/F\) are unramified cyclic of degree \(3\).


\renewcommand{\arraystretch}{1.2}

\begin{table}[ht]
\caption{Frequencies of dihedral \(3\)-class tower groups \(S\) for \(-10^{24}<d_E<0\)}
\label{tbl:TotallyComplexDihedral}
\begin{center}
\begin{tabular}{|c|c|r||c|c|r||r|c||c|r|}
\hline
 \(G\simeq\)              & \(\tau^{(1)}{G}\) & abs. fr.    & \(S\simeq\)              & \(\tau^{(1)}{S}\) & abs. fr.     & \(\lvert d_E\rvert_{\text{min}}\) \\
\hline
 \(\langle 243,5\rangle\) & \(1^2\)           & \(83\,353\) & \(\langle 81,3\rangle\)  & \(21\)            & \(250\,059\) &  \(4\,027^3\)                     \\
 \(\langle 243,5\rangle\) & \(1^2\)           & \(83\,353\) & \(\langle 81,13\rangle\) & \(1^3\)           &  \(83\,353\) &  \(4\,027^3\)                     \\
 \(\langle 243,7\rangle\) & \(1^2\)           & \(41\,398\) & \(\langle 81,3\rangle\)  & \(21\)            &  \(82\,796\) & \(12\,131^3\)                     \\
 \(\langle 243,7\rangle\) & \(1^2\)           & \(41\,398\) & \(\langle 81,13\rangle\) & \(1^3\)           &  \(82\,796\) & \(12\,131^3\)                     \\
\hline
\end{tabular}
\end{center}
\end{table}


\subsubsection{Totally real dihedral fields}
\label{sss:TotallyRealDihedral}
\noindent
The computational information on \(3\)-tower groups \(G:=\mathrm{G}_3^\infty{F}\)
of real quadratic fields \(F\) in Table
\ref{tbl:RealTwoStage9}
admits the purely theoretical deduction
of impressive statistics for \(3\)-tower groups \(S:=\mathrm{G}_3^\infty{E}\)
of totally real dihedral fields \(E\) in Table
\ref{tbl:TotallyRealDihedral}
by means of Theorem
\ref{thm:MaxSbgCc1}
and Theorem
\ref{thm:General}.
Again, we use the innovative result that the groups \(S\triangleleft G\) are maximal subgroups of \(G\),
since the extensions \(E/F\) are unramified cyclic cubic.


\renewcommand{\arraystretch}{1.2}

\begin{table}[ht]
\caption{Frequencies of dihedral \(3\)-class tower groups \(S\) for \(0<d_E<10^{27}\)}
\label{tbl:TotallyRealDihedral}
\begin{center}
\begin{tabular}{|c|c|r||c|c|r||r|c||c|r|}
\hline
 \(G\simeq\)             & \(\tau^{(1)}{G}\) & abs. fr.     & \(S\simeq\)             & \(\tau^{(1)}{S}\) & abs. fr.     & \((d_E)_{\text{min}}\) \\
\hline
 \(\langle 81,7\rangle\) & \(1^2\)           & \(122\,955\) & \(\langle 27,3\rangle\) & \(1^2\)           & \(122\,955\) & \(142\,097^3\)         \\
 \(\langle 81,7\rangle\) & \(1^2\)           & \(122\,955\) & \(\langle 27,4\rangle\) & \(1^2\)           & \(245\,910\) & \(142\,097^3\)         \\
 \(\langle 81,7\rangle\) & \(1^2\)           & \(122\,955\) & \(\langle 27,5\rangle\) & \(1^3\)           & \(122\,955\) & \(142\,097^3\)         \\
\hline
\end{tabular}
\end{center}
\end{table}


\noindent
The first row of Table
\ref{tbl:TotallyRealDihedral}
reveals extensive realizations of the extraspecial group \(S=\langle 27,3\rangle\)
as \(3\)-tower group of dihedral fields.
This is the first time that \(S=\langle 27,3\rangle\) occurs as a \(3\)-tower group.
It is forbidden for quadratic fields,
and it did not occur for cyclic cubic fields and bicyclic biquadratic fields, up to now.

\begin{theorem}
\label{thm:NewRealization}
\textbf{(A new realization as \(3\)-tower group.)}
The extraspecial \(3\)-group \(S=\langle 27,3\rangle\) of coclass \(1\) and exponent \(3\)
occurs as \(3\)-class tower group \(\mathrm{G}_3^\infty{E}\) of totally real dihedral fields \(E\) of degree \(6\).
\end{theorem}

\begin{proof}
The group \(S=\langle 27,3\rangle\) possesses the relation rank \(d_2{S}=4\).
According to the Shafarevich Theorem,
it is therefore excluded as \(3\)-tower group \(\mathrm{G}_3^\infty{F}\)
of both, imaginary and real quadratic fields \(F\).
However, the combination of Theorem
\ref{thm:MaxSbgCc1}
and Theorem
\ref{thm:General}
proves its occurrence as \(3\)-class tower group \(\mathrm{G}_3^\infty{E}\)
of totally real dihedral fields \(E\) of degree \(6\),
as visualized in Table
\ref{tbl:TotallyRealDihedral}.
\end{proof}


\begin{theorem}
\label{thm:Dihedral}
\textbf{(\(3\)-class tower groups of totally real dihedral fields.)}
Let \(F=\mathbb{Q}(\sqrt{d})\) be a real quadratic field
with \(3\)-class group \(\mathrm{Cl}_3{F}\simeq C_3\times C_3\)
and fundamental discriminant \(d>1\).
Suppose the second order Artin pattern
\(\mathrm{AP}^{(2)}{F}=(\tau^{(2)}(F),\varkappa^{(2)}(F))\)
is given by
the abelian type invariants \(\tau^{(2)}(F)=\lbrack 1^2;(2^2,1^2,1^2,1^2)\rbrack\)
and the transfer kernel type \(\varkappa^{(2)}(F)=\lbrack 1;(0000)\rbrack\).
Let \(E_2,E_3,E_4\) be the three unramified cyclic cubic relative extensions of \(F\)
with \(3\)-class group \(\mathrm{Cl}_3{E_i}\simeq C_3\times C_3\).

Then \(E_i/\mathbb{Q}\) is a totally real dihedral extension of degree \(6\),
for each \(2\le i\le 4\),
and the connection between the component \(\#\varkappa^{(3)}(F)_i=\#\ker(T_{E_i,F_3^{(1)}})\)
of the third order transfer kernel type \(\varkappa^{(3)}(F)\)
and the \(3\)-class tower group \(S_i=G_3^\infty{E_i}=\mathrm{Gal}((E_i)_3^{(\infty)}/E_i)\) of \(E_i\)
is given in the following way:
\begin{equation}
\label{eqn:Dihedral}
\begin{aligned}
 & \#\varkappa^{(3)}(F)_i=3 & \Longleftrightarrow & \quad S_i\simeq\langle 243,27\rangle & \text{ with } \varkappa(S_i)=(1000), \\
 & \#\varkappa^{(3)}(F)_i=9 & \Longleftrightarrow & \quad S_i\simeq\langle 243,26\rangle & \text{ with } \varkappa(S_i)=(0000).
\end{aligned}
\end{equation}
\end{theorem}

\begin{proof}
This theorem was expressed as a conjecture in
\cite{Ma17,Ma17b},
and is now an immediate consequence of Theorem
\ref{thm:MaxSbgCc1}.
\end{proof}

\begin{remark}
\label{rmk:Dihedral}
Recall that each unramified cyclic cubic relative extension \(E_i/F\), \(1\le i\le 4\),
gives rise to a dihedral absolute extension \(E_i/\mathbb{Q}\) of degree \(6\),
that is an \(S_3\)-extension
\cite[Prp. 4.1, p. 482]{Ma1}.
For the trailing three fields \(E_i\), \(2\le i\le 4\),
in the stable part of \(\tau^{(2)}(F)=\lbrack 1^2;(2^2,1^2,1^2,1^2)\rbrack\),
i.e. with \(\mathrm{Cl}_3{E_i}\simeq C_3\times C_3\),
we have constructed the unramified cyclic cubic extensions \(\tilde{E}_{i,j}/E_i\), \(1\le j\le 4\),
and determined the Artin pattern \(\mathrm{AP}^{(2)}{E_i}\) of \(E_i\),
in particular, the transfer kernel type of \(E_i\) in the fields \(\tilde{E}_{i,j}\)
of absolute degree \(18\).
The dihedral fields \(E_i\) of degree \(6\) share a common polarization
\(\tilde{E}_{i,1}=F_3^{(1)}\), the Hilbert \(3\)-class field of \(F\),
which is contained in the relative \(3\)-genus field \((E_i/F)^\ast\),
whereas the other extensions \(\tilde{E}_{i,j}\) with \(2\le j\le 4\)
are non-abelian over \(F\), for each \(2\le i\le 4\).
Our computational results underpin Theorem
\ref{thm:Dihedral}
concerning the infinite family of totally real dihedral fields
\(E_i\) for varying real quadratic fields \(F\).
\end{remark}



\section{Acknowledgements}
\label{s:Acknowledgements}

\noindent
The author gratefully acknowledges that his research was supported by the
Austrian Science Fund (FWF): P 26008-N25.




\end{document}